\documentclass[11pt,fleqn]{article}
 \title{Complements on Furtw\"angler$'$s 
second theorem \\ \, and Vandiver$'$s cyclotomic integers} 
\author{Roland Qu\^eme}
\usepackage{amsmath,amsthm,amstext,amsbsy,eufrak}
\newtheorem{thm}{Theorem}[section]
\newtheorem{cor}[thm]{Corollary}

\newtheorem{lem}[thm]{Lemma}
\newtheorem{conj}{Conjecture}
\newtheorem{defi}{Definition}
\newtheorem{notas}{Notations}
\newtheorem{rema}{Remark}
\font\mathbb=msbm10

\newcommand{\F}{\mbox{\mathbb F}}

\newcommand{\Q}{\mbox{\mathbb Q}}
\newcommand{\Z}{\mbox{\mathbb Z}}

\newcommand{\be}{\begin{equation}}
\newcommand{\ee}{\end{equation}}
\newcommand{\bd}{\begin{displaymath}}
\newcommand{\ed}{\end{displaymath}}
\newcommand{\bn}{\begin{enumerate}}
\newcommand{\en}{\end{enumerate}}

\newcommand{\mk}{\mathfrak}
\newcommand{\ml}{\mathcal}
\newcommand{\mf}{\mathbf}

\newcommand{\ov}{\bar}
\date{2011 Nov  21}

\begin{document}
\maketitle
\tableofcontents
\maketitle
\begin{abstract}
This article  deals with  a conjecture generalizing  the second case of 
Fermat's Last Theorem, called $SFLT2$ conjecture:
{\it Let $p>3$ be a prime, $K:=\Q(\zeta)$
the $p$th cyclotomic field and $\Z_K$ its ring of integers.
The diophantine  equation 
$(u+v\zeta)\Z_K=\mk w_1^p$,
with  $u,v\in\Z\backslash\{0\}$ coprime,  $uv\equiv 0 \mod p$ 
and  $\mk w_1$  
ideal of  $\Z_K$, has no solution.}
Assuming that $SFLT2$ fails for $(p,u,v)$, 
let  $q$ be an odd prime not dividing  $uv$,  $n$ the order of 
$\frac{v}{u}\mod q$, $\xi$ a primitive
$n$th
root of unity and $M:=\Q(\xi,\zeta)$.
The aim of this   complement  of the article [GQ]  
of G. Gras and R. Qu\^eme on the same topic, is
to exhibit some strong  properties of the decomposition of  the primes
$\mk Q$  of $\Z_M$ 
over $q$ in certain Kummer $p$-extensions of the field $M$ with the aim 
to derive from them   a conjecture implying SFLT2 and  a  weak  conjecture
implying   
that the  SFLT2 equation can always take    the    reduced form  $u+\zeta
v\in K^{\times p}.$

\end{abstract}
\begin{abstract}

Cet article  traite d'une conjecture g\'en\'eralisant  le second cas 
du dernier th\'eor\`eme de Fermat, ci-apr\`es conjecture $SFLT2$:
{\it   Soit  un nombre premier $p>3$ , $K:=\Q(\zeta)$ le $p$-i\`eme corps
cyclotomique
 et  $\Z_K$ son anneau d'entiers. L'\'equation  diophantienne
$(u+v\zeta)\Z_K=\mk w_1^p$,
avec   $u,v\in\Z\backslash\{0\}$ premiers entre eux,  $uv\equiv 0\mod  p$
et   $\mk w_1$ 
id\'eal entier de   $\Z_K$,  n'a pas de  solution.}
Supposant   SFLT2  fausse pour $(p,u,v)$, 
soient $q$ un  nombre premier ne divisant pas $uv$, 
$n$  l'ordre  de $\frac{v}{u}\mod  q$, $\xi$ une racine  primitive
$n$-i\`eme
de l'unit\'e et  $M:=\Q(\xi,\zeta)$.
L' objectif de ce compl\'ement de  l'article [GQ] de G. Gras et R. Qu\^eme
sur le m\^eme sujet est 
de mettre en \'evidence certaines propri\'et\'es fortes de
d\'ecomposition des
id\'eaux premiers $\mk Q$ de $\Z_M$ au  dessus de  $q$ 
dans  certaines  $p$-extensions  de Kummer du corps  $M$ avec l'objectif  
d'en d\'eduire   une  conjecture impliquant SFLT2 et  une  conjecture
faible    impliquant  que l'\'equation SFLT2 
peut toujours se mettre sous la forme 
r\'eduite $u+\zeta v\in K^{\times p}$.
\footnote{

- {\it Keywords:} Fermat's Last Theorem, cyclotomic fields, cyclotomic
units, 
class field theory, Vandiver's and Furtw\"angler's theorems

- {\it AMS subject classes :} 11D41, 11R18, 11R37}
\end{abstract}

\maketitle
\section{  Introduction } \label{s1012111}

In all this article we denote by $\zeta$ the $p$th primitive root of unity 
defined by $\zeta:=e^{\frac{2\pi i}{p}}$.
In [Gr2, Conj.\,1.5],  G. Gras  has given  a conjecture  which
implies Fermat's Last Theorem (FLT): we recall here this conjecture which
will be called {\it Strong Fermat's Last Theorem} (SFLT).

\begin{conj}\label{cj1}  Let $p$ be an odd prime,  set $K = \Q(\zeta)$ 
and ${\mathfrak p} = (\zeta-1)\,\Z[\zeta]$. Then the equation 
    $$(u+v \,\zeta)\,\Z[\zeta] = {\mathfrak p}^\delta\, {\mathfrak w}_1^p$$
    in coprime integers $u,\,v$, where  $\delta$ is any integer $\geq 0$ 
    and  ${\mathfrak w}_1$ is any integral ideal of $K$, 
    has no solution for $p>3$ except the trivial  ones for which
   $u+v \,\zeta = \pm 1$, $\pm \zeta$, $\pm (1+\zeta)$, 
   or  $\pm (1-\zeta)$.
    \end{conj}

The cases $uv(u+v)\not\equiv 0\mod p, \ uv\equiv 0\mod p$, and $u+v\equiv
0\mod p$
are called respectively the first, second, and special case of $SFLT$.

\smallskip From some works of P. Furtw\"angler and H.S. Vandiver,  
G. Gras and R. Qu\^eme  [GQ] have put
the basis of a new cyclotomic approach to  Fermat's Last Theorem  by 
introducing some auxiliary  fields of the form $\Q(\mu_{q-1})$ 
with prime  $q\not= p$   to  study  $SFLT$ equation  and  derive some
consequences for $FLT$. 

\smallskip 

In this   article,  we  examine    some particularities of   the  second
case of $SFLT$
(hereinafter $SFLT2$). 
Without loss of generality in the context of this work,  we   choose the
following  formulation of   $SFLT2$ in  the sequel:

{\it Let $p>3$ be a prime, $K:=\Q(\zeta)$
the $p$th cyclotomic field and $\Z_K$ its ring of integers.
The diophantine  equation 
$(u+v\zeta)\Z_K=\mk w_1^p$,
with  $u,v\in\Z\backslash\{0\}$ coprime,  $v\equiv 0 \mod p$ 
and  $\mk w_1$  
ideal of  $\Z_K$, has no solution.}

Thus it is always assumed in the sequel, without further mention, that $p|v$.
Observe that $SFLT2$ implies the second  case  $FLT2$ of $FLT$.

\subsection{General definitions and notations } \label{definitions}
At first, we fix some general notations, conventions  and  definitions,  the 
context of Fermat's Last Theorem and then we explain the aims of the
article.

\begin{notas}\label{n1}
$ $
{\rm

- Let  $p>3$ be a prime, $\zeta:=e^{\frac{2\pi i}{p}}$, $K:=\Q(\zeta)$
the  $p$th cyclotomic number field, $\Z_K$ the ring of integers  of $K$, 
 and  $\mk p=(1-\zeta)\Z_K$  the prime ideal of $\Z_K$ over $p$.

 - Let $g:={\rm Gal}(K/\Q)$, for $k \not\equiv 0 \mod p$ and  $s_k :
\zeta \rightarrow\zeta^k$   
 the $p-1$ distinct elements of $g$. 

- Let $C\ell_K$, $C\ell$, $C\ell^-$  and $C\ell_{[p]}$ be respectively the class group of $K$, 
the $p$-class group of $K$, the negative part of the $p$-class group of $K$ and  the $p$-elementary
class group of $K$.
For any ideal  $\mk a$  of $K$, 
let us note $c\ell_K(\mk a)$ and $c\ell(\mk a)$ the class of $\mk a$ in $C\ell_K$ and  $C\ell$.
}
\end{notas}

\begin{defi}$ $
A number $\alpha\in K^\times$ prime to $p$, such that $\alpha\Z_K$ is
the $p$th power of an ideal,  is called a pseudo-unit.
The pseudo-unit $\alpha$ is   $p$-primary  (i.e. 
the extension $K(\sqrt[p]{\alpha})/K$ is unramified at $\mk p$) 
if and only if $\alpha$ is congruent 
to a $p$-power $\mod \mk p^p$, see [Gr2] lem 2.1.
\end{defi}

If $SFLT2$ fails for  $(p,u,v)$ with $p|v$ then $\gamma := u+\zeta v
\in\Z_K$
is a $p$-primary pseudo-unit of $\Z_K$ since $\gamma \equiv u \mod p$
(a generalization of a result of Kummer given  again in [Gr2], Theo.\,2.2).

 We will derive from our results on $SFLT2$ some consequences for $FLT2$.
We adopt  in the sequel the following notations for  an hypothetic
counterexample to $FLT2$
$$x^p+y^p+z^p=0,$$
 $x,y,z\in\Z\backslash\{0\}$ coprime and $p|y$.
It would be a counterexample to $SFLT2$ with $(u,v)=(x,y)$ (resp.
$(u,v)=(z,y)$), 
verifying  moreover 
$x+y=z_0^p$  for some $z_0 \vert z$ (resp. $z+y=x_0^p$ for some $x_0\vert
x$).

\medskip

\begin{notas}
{\rm 
 
We adopt the following notations:

  - For any integer  $m>0$,  let $\Phi_m(X)$ be 
the $m$th cyclotomic polynomial and   $\phi(m)$ the Euler indicator. 
For   $a,b\in\Z\backslash \{0\}$, let us define  
$\Phi_{m}(a,b) := b^{\phi(m)}\Phi_{m}\Big(\frac{a}{b}\Big).$
Clearly $\Phi_m(a,b)\in\Z[a,b]$.

\smallskip

- Recall that we assume  that $SFLT2$ fails for  $(p, u, v)$, so  with $u,v$   relatively
prime and  $v \equiv 0 \mod p$.
Recall  [GQ, Lem.\,2]: {\it let $n\geq 1$ and $q$ be a prime. 
The following properties are equivalent:

\smallskip (i) $q\not|\  n$ and $q \ |\ \Phi_n(u,v)$.

\smallskip (ii) $q\not|\  uv$ and $\frac{v}{u}$ is of order $n\mod q$.}

\smallskip

-  Let  $q$ be a prime number dividing $\Phi_n(u,v)$ with $q\ \not|\ n$   and $n=dp^r$  
where $d$ is prime to $p$ and $r\geq 0$.
$$\Phi_{d\,p^r}(u,v) := \prod\  (u\,\psi^i\,\zeta_r^j - v)
\mbox{\ \ for all\ \ } i\ \in (\Z/d\Z)^\times \mbox{\ \ and\ \ } j\ \in
(\Z/p^r\Z)^\times,$$
where $\psi:=e^{\frac{2\pi i}{d}}$   and $\zeta_r:=e^{\frac{2\pi
i}{p^r}}$\  
(observe that the two previous definitions    
$\zeta:=e^{\frac{2\pi i}{p}}$ and $\zeta_r:=e^{\frac{2\pi i}{p^r}}$
imply that $\zeta_1=\zeta$).

\smallskip
 Let us  fix the  root of unity   $\xi :=\psi \,\zeta_r$.  
 \smallskip
Let  $L:= \Q(\xi)$ and $M=LK = \Q(\xi,\zeta)$.
Put ${\mathfrak q}  = (q, u\,\xi - v)$ where $\mk q$ is a prime ideal of
$L$ over $q$  because we have assumed $q\not|\  n$.
\footnote{Observe that the  prime ideal $\mk q$   
is fixed unambiguously by this  choice of $\xi$. } 

\smallskip
(i)  If $r=0$ then $L= \Q(\psi)$ and $M$ is of degree $p-1$ over
 $L$. We denote by ${\mathfrak Q} $ any prime ideal of $M$ above 
 ${\mathfrak q}$ and by $\mk q_K$ the prime ideal of $K$
 under  $\mk Q$;

\smallskip
(ii)  If $r \geq 1$ then $M=L$ and 
${\mathfrak Q} ={\mathfrak q}$.}
 
\end{notas}

\smallskip
  Let us recall the definition of the $p$th power 
residue symbols in $K$ and $M$  with values in  $\mu_p$ (see [GQ]
definition 2).

\begin{defi}\label{d1103061}
If $\alpha \in M$ is prime to
${\mathfrak Q}\,\vert\, {\mathfrak q}$ in $M$, then let $\ov\alpha$ 
be the image of $\alpha$ in
the residue field $\Z_M/{\mathfrak Q} \simeq \mf F_{q^f}$; 
since
$\zeta \in \Z_M$, the image $\ov \zeta$ of $\zeta$ is of order $p$
(since $\zeta \not\equiv 1 \mod {\mathfrak Q}$) and we can put
$\ov\alpha_{}^{\,\kappa} = \ov \zeta^{\,\mu}$, $\kappa=\frac{q^f-1}{p}$,
$\mu \in \Z/p\Z$,
which defines the $p$th power residue symbol
$\Big(\frac{\alpha}{{\mathfrak Q}}\Big)_{\!\!M} :=
\zeta^\mu$; this symbol is equal to 1 if and only if $\alpha$ is a
local $p$th power at ${\mathfrak Q}$ (see
[Gr1,\,I.3.2.1, Ex.\,1]).

\smallskip\noindent
With this definition, for any automorphism $\tau\in {\rm Gal}(M/\Q)$ one
obtains, from $\alpha_{}^{\,\kappa} \equiv \zeta^{\,\mu}
\mod {\mathfrak Q}$,
$\tau \alpha_{}^{\,\kappa}  \equiv \tau \zeta^\mu
\mod \tau {\mathfrak Q}$, thus
$\tau \Big(\frac{\alpha}{{\mathfrak Q}}\Big)_{\!\!M} =
\Big(\frac{\tau \alpha}{\tau {\mathfrak Q}}\Big)_{\!\!M} =
\tau \zeta^\mu$.

\smallskip\noindent
If $\alpha \in K$, since ${\mathfrak q}_K\,\vert\,q$ in $K$ splits
totally in $M/K$, we have $\Z_K/{\mathfrak q}_K \simeq \Z_M/{\mathfrak Q}$
and
$\Big(\frac{\alpha}{{\mathfrak q}_K}\Big)_{\!\!K} =
\Big(\frac{\alpha}{{\mathfrak Q}}\Big)_{\!\!M}$
for any ${\mathfrak Q}\,\vert\,{\mathfrak q}_K$.
In particular this implies $\Big(\frac{\zeta}{{\mathfrak 
q}_K}\Big)_{\!\!K} = \zeta^{\kappa}$
(the symbol of $\zeta$ does not depend on the choice of ${\mathfrak 
q}_K \,\vert\, q$).
\end{defi}

\subsection {The aim of the article}

In the $FLT2$ context defined before,  Furt\-w\"angler proved 
that,   if $q$ is a prime such that  $q\ |\ x(x+y)$,  
then $q^{p-1}\equiv 1\mod  p^2$
(first Furtw\"angler theorem for Fermat, see Furtw\"angler [Fur] and 
Ribenboim [Rib1, 3B]), and 
if $q$ is a prime such that $q\ |\ x-y$ then $q^{p-1}\equiv 1\mod  p^2$
(second  Furtw\"angler theorem for Fermat, see Furtw\"angler [Fur] and 
Ribenboim [Rib1, 3C]. 

For a generalization of the  Furtw\"angler results  
in the SFLT2 context with $q|u(u^2-v^2)$, see  Gras [Gr2], Theo.\,3.20
and also Gras and Qu\^eme [GQ], Cor.\,2 and 3.

\begin{defi}
A prime $q$ is said {\it  $p$-principal}
if the class $c\ell_K (\mk q_K)\in C\ell_K$  of any prime ideal 
$\mk q_K$ of $\Z_K$ above  $q$ is the $p$th power of a class, which is 
equivalent to $\mk q_K = {\mk a}^p (\alpha)$, for an ideal ${\mk a}$
of $K$ and an $\alpha \in K^\times$.
This contains the case where  the class $c\ell_K (\mk q_K)$ is of order
coprime with  $p$.
\end{defi}

We  suppose that  SFLT2 fails for    $(p,u,v)$.
Let  $q$ be a $p$-principal prime with $\frac{v}{u}$ of order
$n\mod q$ with $q\not|\ n$,  so such that  $q\ |\ \Phi_n(u,v)$. 
Renewing some  ideas of Vandiver for $FLT$ in [Va1, Va2]  involving 
the systematic use of   the $p$th  power residue symbols 
$\Big(\frac{a}{\mk Q}\Big)_{\!\! M}$ for  $a\in M$ 
coprime with $\mk Q$ (see definition \ref{d1103061}),  
the aims of this article are: 
\bn
\item
Exhibit some strong  properties of the decomposition of
 the prime  $q$ in certain Kummer $p$-extensions of the field $M$
involving the so-called {\it Vandiver cyclotomic units} and  more
particularly 
in detailing the case of   the primes $q$ dividing $u^{2p}-v^{2p}$.
\item  Set a    conjecture in contradiction with this  decomposition
for  some {\it arbitrarily large  }  primes  $q$  implying   
 the SFLT2 conjecture. 
\item  Set a   weak conjecture in contradiction with this  decomposition
for  some {\it small}  primes  $q$  implying   
that the SFLT2 equation could be   reduced  to the form   $$u+\zeta v\in
K^{\times p}.$$ 
This reduced  form   seems  possibly   easier to search 
for a proof of   SFLT2 conjecture with a diophantine approach.
\item Set a weak  conjecture in contradiction with such decompositions 
of  these primes  $q$  implying   the generalization  to the second case
FLT2   of  the  Terjanian  FLT1 theorem [Ter]  
(i.e. $x^{2p}+y^{2p}+z^{2p}=0$ has no solution if $xyz\not\equiv 0\mod p$).
\en

\section{The main theorem}
We give at first a definition and  an elementary lemma independent of the
SFLT conjecture.

\begin{defi}\label{d11O9221}
Let $n=dp^r$, with $d,p$ coprime and $r\geq 0$.
Let $\xi$ be a fixed primitive $n$th  root of unity
$\xi=\psi\zeta_r$ where  $\psi:=e^{\frac{2\pi i}{d}}$ and
$\zeta_r:=e^{\frac{2\pi i}{p^r}}$.

For all $\ 0\leq k <p-1$, let us define\,\footnote{The reason why  
$k = p-1$ is   discarded 
will be explained  in remark \ref{r1p1} after the lemma \ref{l2}.}
$$\varepsilon_k:= 1+\xi\zeta^k. \footnote{$\varepsilon_k$ is used with
this meaning
in the sequel of the article.}$$
\end{defi}

\begin{lem}\label{l1}$ $

a) If $k=0$, $\varepsilon_0 = 1 + \xi$  is a cyclotomic unit of $L$ 
except if $d=1$ ($\varepsilon_0 = 2$) or $d=2$  ($\varepsilon_0 = 0$).

\smallskip    
b) Suppose that $0<k<p-1$.

\smallskip

(i) If $d>2$ then  $\varepsilon_k=1+\xi\zeta^k$ is a cyclotomic unit.

\smallskip

(ii) If $d=2$ then $\varepsilon_k$ is not a cyclotomic unit and 

\smallskip
\hspace{0.6cm}--  If $r \geq 1$ then
$\varepsilon_k=1-\zeta_r^{1+kp^{r-1}}\in\Z[\zeta_r]$
with $\varepsilon_k\Z[\zeta_r]=\mk p_r$ 
where $\mk p_r$ is the prime ideal of $\Z[\zeta_r]$ above $p$.

\smallskip
\hspace{0.6cm}--  If $r=0$ then $\varepsilon_k= 1-\zeta^{k}$ with
$\varepsilon_k\Z_K=\mk p$.

\smallskip

(iii) If $d=1$ then $\varepsilon_k$ is a cyclotomic unit and 

\smallskip
\hspace{0.6cm}-- If $r \geq 1$ then $\varepsilon_k=
1+\zeta_r^{1+kp^{r-1}}$.

\smallskip
\hspace{0.6cm}--  If $r=0$ then $\varepsilon_k= 1+\zeta^{k}$.
\begin{proof}
Left to the reader.
\end{proof}
\end{lem}

\smallskip

The following lemma using Hilbert class field theory for  $K$ plays a
central role in the article.

\smallskip

\begin{lem}\label{l2}
Suppose that $SFLT2$ fails for  $(p,u,v)$ with $p|v$. 
Let  $q\not|\ puv$ be a $p$-principal prime,  
$n$ the order of $\frac{v}{u}\mod q$, $\xi:=e^{\frac{2\pi i}{n}}$  and $\mk
q$ be  the prime ideal 
$\mk q:= (u\xi-v, q)$ of  $\Z_L$.
Then 
$$\Big(\frac{\varepsilon_k}{\mk Q}\Big)_{\!\! M}= \Big(\frac{u}{\mk
q_K}\Big)_{\!\! K}^{\!\!-1}
\mbox{\ for all\ }k=1,\ldots,p-2 \mbox{\  and all \ }\mk Q|\mk q,
\footnote{Observe that $\Big(\frac{u}{\mk q_K}\Big)_{\!\!  K}^{-1}$ does
not depend on $k$ 
and recall that  $\mk Q=\mk q$ if $r>0$.}$$  
where $\varepsilon_k=1+\xi\zeta^k$. 
\begin{proof}$ $

-We have $u\xi- v\equiv 0 \mod \mk q$, so $u\xi - v\equiv 0\mod  \mk Q$
for all $\mk Q|\mk q$,  hence 
with   $\gamma := u+\zeta  v$, we get
$s_k(\gamma)=u+\zeta^k v\equiv u(1+\xi\zeta^k)\equiv u \varepsilon_k\mod  
\mk Q$,
 for all $k\not\equiv 0\mod  p$.

- We obtain 
$\Big(\frac{s_k(\gamma)}{\mk Q}\Big)_{\!\! M}= 
\Big(\frac{u}{\mk Q}\Big)_{\!\! M}
\Big(\frac{\varepsilon_k}{\mk Q}\Big)_{\!\! M},$
so 
$\Big(\frac{s_k(\gamma)}{\mk q_K}\Big)_{\!\! K}= 
\Big(\frac{u}{\mk q_K}\Big)_{\!\! K}
\Big(\frac{\varepsilon_k}{\mk Q}\Big)_{\!\! M}.$

-  The numbers $s_k(\gamma)$ are $p$-primary pseudo-units, which implies 
 $\Big(\frac{s_k(\gamma)}{\mk q_K}\Big)_{\!\! K}=1$  for all $k\not\equiv
0\mod  p$ 
from the Hilbert
 class field decomposition theorem because, by assumption,  $q$ is
$p$-principal.
\end{proof}
\end{lem}

\begin{rema}\label{r1p1}
{\rm We explain why we can discard the value $k=p-1$ of the index $k$.
We exclude the value $k=p-1$, because $\varepsilon_k$ would
be null if and only if  $d=2,r=1$ and $k=p-1$.
For all the other $(d,r,k),\ 1\leq k\leq p-1$ with  $\varepsilon_k\not=0$,
we will show that it is always possible to
express $\Big(\frac{1+\xi\zeta^{p-1}}{\mk Q}\Big)_{\!\! M}$ in function of
$\Big(\frac{1+\xi\zeta^{k}}{\mk Q}\Big)_{\!\! M},\ k=1,\dots,p-2$ as 
follows:
we start from $$u(1+\xi\zeta^j)\equiv s_j(\gamma)\mod\mk q,\mbox{\ for\ }
j=1,\dots,p-1,$$
where $1+\xi\zeta^j$ is always
nonzero, therefore 
$$u^{p-1}\prod_{j=1}^{p-1} (1+\xi\zeta^j)\equiv N_{K/\Q}(\gamma)=w_1^p\mod
\mk q,$$ 
so 
$$\Big(\frac{u^{p-1}(1+\xi\zeta)\ldots(1+\xi
\zeta^{p-2})(1+\xi\zeta^{p-1})}{\mk Q}\Big)_{\!\! M}=1
\ \mbox{\ for all\ } \mk Q|\mk q$$
so, from lemma \ref{l2} applied for all $1\leq k\leq p-2$ we get
$\Big(\frac{u(1+\xi\zeta^{p-1})}{\mk Q}\Big)_{\!\! M}=1,$
and thus $$\Big(\frac{1+\xi\zeta^{p-1}}{\mk Q}\Big)_{\!\!
M}=\Big(\frac{1+\xi\zeta^k}{\mk Q}\Big)_{\!\! M},
\ \mbox{\ for all\ } k=1,\dots,p-2.$$}
\end{rema}

\begin{rema} {\rm Suppose moreover  that $FLT2$ fails  for  $(p,x,y,z)$.
In that case, we have (with $x,y$ in place of $u,v$)  the relations
$(x+\zeta y)\Z_K=\mk z_1^p$ where $\mk z_1$ is an ideal of $\Z_K$,
$x+y=z_0^p$ 
and $\frac{x^p+y^p}{x+y}= z_1^p$ with $z = - z_0 z_1$.
We get $x\xi-y\equiv 0\mod  \mk q$, so $x(1+\xi)\equiv z_0^p\mod \mk q$,
thus
$\Big(\frac{x}{\mk q_K}\Big )_{\!\!K}\Big(\frac{1+\xi}{\mk Q}\Big)_{\!\!
M}= 1.$

From $x^p+y^p+z^p=0$, we get $x^p(1+\xi^p)\equiv -z^p\mod\mk q$, so
$\Big(\frac{1+\xi^p}{\mk Q}\Big)_{\!\! M}=1$
and thus  we improve slightly the  lemma \ref{l2} if $FLT2$ fails 
for  $(p,x,y,z)$}:
\end{rema}

\begin{cor}\label{c1}
 Suppose that $FLT2$ fails for $(p,x,y,z)$. Let $q\not|\ pxy$  be a
$p$-principal 
prime   with $\frac{x}{y}$ of order  $n\mod q$.   Then  
$$ \Big(\frac{1+\xi^p}{\mk Q}\Big)_{\!\! M}=1$$ and 
$$\Big(\frac{\varepsilon_k}{\mk Q}\Big)_{\!\! M}= \Big(\frac{x}{\mk 
q_K}\Big)_{\!\! K}^{-1} 
\mbox{\ for all\ }  k=0,1,\dots,p-2\mbox{\  and all\ }\mk Q|\mk q.$$ 
\end{cor}

\subsection {The general case}

We give  the main theorem characterizing  the decomposition of $\mk Q$
in a Kummer $p$-extension of $M$ defined from the Vandiver's cyclotomic
units.

\begin{thm}\label{t1}
Suppose that $SFLT2$ fails for  $(p,u,v)$, $p\geq 3$ with $p|v$. Let $q
\not|\ puv$  be a $p$-principal
prime, $n$ the order of  $\frac{v}{u}\mod q$ and  $\xi:=e^{\frac{2\pi
i}{n}}$.
For all $0\leq k <p-1$,  let 
$\varepsilon_k:= 1+\xi\zeta^k$ and  the  prime ideal    ${\mathfrak q}  :=
(q, u\,\xi - v)$ of $\Z[\xi]$ above $q$.

Then all the prime ideals   $\mk Q$ of $\Z[\xi,\zeta]$ dividing $\mk q$
totally split 
in   the Kummer extension 
$$M \big(\sqrt[p]{<\varepsilon_k\varepsilon_1^{-1}>_{k=1,\ldots,p-2}}\,
\big) \big/M.$$

\begin{proof} It is a reformulation of the previous lemma \ref{l2}
which had shown  that $\Big(\frac{\varepsilon_k}{\mk Q}\Big)_{\!\! M}=
\Big(\frac{\varepsilon_1}{\mk Q}\Big)_{\!\! M}$, for $k=1,\dots,p-2$.
\end{proof}
\end{thm}

\begin{rema} $ $

\smallskip

{\rm 
\bn
\item
Observe  from lemma \ref{l1} that if $d>2$ then  $\varepsilon_{p-1}$ is a
cyclotomic unit. 
If $p=3$ then the 
 SFLT2 equation has actually  infinitely many solutions with $p|v$ of form
$u+j v= (s+jt)^3$,  with $s,t\in\Z$ coprime and $s+t\not\equiv 0\mod 3$,
see[GQ] remark 2.3.
Let $q>3$ be a prime. For any  $s,t$ given, we get  $u=s^3+t^3-3st^2, v=
3s^2t-3st^2$ and the order  
$n=d\cdot3^r$ of $\frac{v}{u}\mod q$. If $d>2$ then $(1+\xi j)(1+\xi
j^2)\not=0$ and  the application of theorem \ref{t1} gives 
$\Big(\frac{(1+\xi j^2)(1+\xi j)^{-1}}{\mk Q}\Big)_{\!\! M}=1.$ We can
verify it  directly  because that case corresponds to   $u+ v j =
(s+jt)^3$.
\item 
 If $SFLT2$ fails  for  $(p,u,v)$ with $p>3$  {\it large} and $p|v$,  and
if the $p$-principal prime    $q\ll p$  ({\it very small} 
compare to $p$)  , the probability is very small  that   the ideal $\mk q$
of $L$ over $q$  split totally  in the Kummer extension 
$M \big(\sqrt[p]{<\varepsilon_k\varepsilon_1^{-1}>_{k=1,\ldots,p-2}}\,
\big) \big/M$: let $p^\delta$ be the degree of this 
Kummer extension;  the probability estimate  that $\mk q$ split totally in
this extension should be roughly  $\mathcal P<\mathcal
O(\frac{\phi(q)}{p^\delta})$.
\item
As a  consequence, if  $SFLT2$  fails for   $(p,u,v)$  with $p|v$
sufficiently large,  then 
these probability estimates   suggest  that the integer $|v|$ should be
an 
{\it ''apocalyptically'' large  integer} with a very large number of
$p$-principal primes $q|v$ with $\kappa\not\equiv 0\mod p$,
which could bring some tools for  another  diophantine tackling of SFLT2
problem.
\en}
\end{rema}

At this stage it is possible to formulate the following conjecture:
\begin{conj}\label{cj1109041} If,   
for all triples  $(p, u,v)\in\Z^3$ with $p>3$  prime,   with $u,v $
coprime and   $p|v$,  
there exists at least 
 one $p$-principal prime $q$ 
with   $puv\not\equiv 0\mod q$, such that  the prime ideal  $\mk q=(q,u\xi-v)$
 is not totally split  in the extension 
$M \big(\sqrt[p]{<\varepsilon_k\varepsilon_1^{-1}>_{k=1,\ldots,p-2}}\,
\big) \big/M,$ 
with  $n$    order of $\frac{v}{u}\mod q$ and  $\xi=e^{\frac{2\pi
i}{n}}$, then $SFLT2$ holds.
\footnote{Note that  we use this formulation, abuse of language for: 
 {\it all the  prime
ideals $\mk Q$ of $M$ over $\mk q$ are not totally split in the extension
$M \big(\sqrt[p]{<\varepsilon_k\varepsilon_1^{-1}>_{k=1,\ldots,p-2}}\,
\big) \big/M.$}
This remark concerns also several statements of conjectures or theorems
in the sequel.} 
\end{conj}

\begin{rema}$ $
{\rm 
\bn
\item
Note that the formulation of this conjecture set for all pairs $(u,v)$
coprime with $p|v$ does not assume  that $u+\zeta v$ is a pseudo-unit (or
not), and thus  is independent of
 the SFLT problem. 
In this conjecture the assumption $puv\not\equiv 0\mod q$ can  require
that $q$ be taken   arbitrarily large, so    with $q\gg p$ (very large
compare to $p$). 
 For a more general 
conjecture of similar nature implying SFLT (first, second and special
case), see [GQ] conjecture 3.
\item
In subsection \ref{snp}, we will   show the existence of   {\it effective
small}   non $p$-principal
primes $q$  compare to $p$ ($q\ll p$) with $ uv\not\equiv 0\mod p$  that
allows to reduce the $SFLT2$ equation to the weaker 
form of diophantine equation  (possibly  easier to tackle)
$$u+\zeta v\in K^{\times p}.$$ 
\en 
}
\end{rema}

\subsection {The case  $n\in\{p,1,2p, 2\}$}

Recall that we have assumed $p>3$. In  this  subsection,  we suppose that
$SFLT2$ 
fails for  $(p,u,v)$ and  
we apply the lemma \ref{l2} in  fixing  $n\in\{p,1,2p, 2\}$ to
derive some strong properties of all the $p$-principal primes $q$ dividing
$\Phi_n(u,v)$ for these values of $n$.
Observe that  we have  $M=K$ in all these  cases.

\subsubsection{The two cases $n= p$  and $n=1$.}
The reunion of these  two cases  allows us to investigate the properties
of 
all  the  $p$-principal primes $q$ dividing $u^p-v^p$.
\begin{cor}\label{c2} 
Case $n=p$: suppose that $SFLT2$ fails for  $(p,u,v)$. 
If    $q$ is  a  $p$-principal  prime dividing $\frac{u^p-v^p}{u-v}$  then 
\begin{eqnarray*}
&& q\equiv 1\mod  p^2, \\
&& (1+\zeta)^{(q-1)/p}\equiv u^{(q-1)/p}\equiv v^{(q-1)/p}\equiv
2^{(q-1)/p}\equiv 1\mod  q,\\
\end{eqnarray*}
and the prime ideal $\mk q_K:= (q,u\zeta-v)$
splits totally in the extension 
$$K \big(\sqrt[p]{<1+\zeta^j>_{j=0,1,\ldots,p-2}}\, \big)/K.$$

\begin{proof}
Here $\xi=\zeta$, $u\zeta-v\equiv 0\mod \mk q$, 
$\varepsilon_k=1+\zeta^{k+1}$, $M=L=K$ and $\mk q=\mk Q=\mk q_K$, so  
$\Big(\frac{1+\zeta^{k+1}}{\mk q_K}\Big)_{\!\! K}=\Big(\frac{u}{\mk
q_K}\Big)_{\!\! K}^{-1}
\ \ \mbox{\ for all\ } k=1,\dots,p-2$.
It follows that  
$\Big(\frac{1+\zeta^{2}}{\mk q_K}\Big)_{\!\! K}= 
\Big(\frac{1+\zeta^{-2}}{\mk q_K}\Big)_{\!\! K}$,
which implies that   $\Big(\frac{\zeta}{\mk q_K}\Big)_{\!\! K}=1$,
thus $q\equiv 1\mod  p^2$, observing that $q\equiv 1\mod  p$. 

\smallskip
 From  $\Big(\frac{\zeta}{\mk q_K}\Big)_{\!\! K}=1$, we get 
$\Big(\frac{1+\zeta}{\mk q_K}\Big)_{\!\! K}= 
\Big(\frac{1+\zeta^{p-1}}{\mk q_K}\Big)_{\!\! K}$,
so $\Big(\frac{1+\zeta^{j}}{\mk q_K}\Big)_{\!\! K}=\Big(\frac{u}{\mk
q_K}\Big)_{\!\! K}^{-1}$,
 for all $j=1,\dots,p-1$.
Thus  
$\Big(\frac{1+\zeta}{\mk q_K}\Big)_{\!\! K}=\ldots = 
\Big(\frac{1+\zeta^{p-1}}{\mk q_K}\Big)_{\!\! K}$.
Therefore 
$\Big(\frac{N_{K/\Q}(1+\zeta)}{\mk q_K}\big)_{\!\!}=
\Big(\frac{u^{-1}}{\mk q_K}\Big)_{\!\! K}^{p-1})=1$, 
so $\Big(\frac{u}{\mk q_K}\Big)_{\!\! K}=\Big(\frac{v}{\mk q_K}\Big)_{\!\!
K}=1,$
and gathering these results we get 
$$\Big(\frac{1+\zeta}{\mk q_K}\Big)_{\!\! K}=\ldots =
\Big(\frac{1+\zeta^{p-1}}{\mk q_K}\Big)_{\!\! K}=\Big(\frac{u}{\mk
q_K}\Big)_{\!\! K}=
\Big(\frac{v}{\mk q_K}\Big)_{\!\! K}=1.$$

\smallskip

From $u\zeta-v\equiv 0\mod \mk q_K$,  
we have $w_1^p=\frac{u^p+v^p}{u+v}\equiv \frac{2u^p}{u(1+\zeta)}\mod \mk
q_K$, 
so $\Big(\frac{2}{\mk q_K}\Big)_{\!\! K}=1$,
and  finally 
$$\Big(\frac{2}{\mk q_K}\Big)_{\!\! K}=\Big(\frac{1+\zeta}{\mk
q_K}\Big)_{\!\! K}=\dots =
\Big(\frac{1+\zeta^{p-1}}{\mk q_K}\Big)_{\!\! K}=\Big(\frac{v}{\mk
q_K}\Big)_{\!\! K}
=\Big(\frac{u}{\mk q_K}\Big)_{\!\! K}=1. $$ 

By  conjugation by $s_\ell$,  we get 
$\Big(\frac{1+\zeta}{s_\ell(\mk q_K)}\Big)_{\!\!K}=1$  for any
$\ell\not\equiv 0\mod  p$, thus  
$$(1+\zeta)^{(q-1)/p}\equiv 1\mod  q, $$
and finally   $q$ splits totally in the extension   
$K \big(\sqrt[p]{<1+\zeta^j>_{j=0,1,\ldots,p-1}}\, \big)/K$. 
\end{proof}
\end{cor}

\begin{rema}
{\rm 
 See theorem \ref{t1107232} for a strong  generalization of this result
{\it without assumption} that $q$ is a $p$-principal prime  when FLT2 fails for
$(p,x,y,z)$ with
$(u,v)=(x,y)$. 
}
\end{rema}

\begin{cor}\label{c3} 
Case $n=1$: suppose that $SFLT2$ fails for  $(p,u,v)$. 
If    $q$ is    a  $p$-principal prime  of order $f\mod  p$, 
$q$ divides $u-v$
and $\mk q_K$    is any prime ideal of $\Z_K$ over $q$
then we have:
\begin{eqnarray*}
&& q^f\equiv 1\mod  p^2,\\
&&  u^{(q^f-1)/p}\equiv v^{(q^f-1)/p}\equiv (1+\zeta)^{(q^f-1)/p}\equiv   2^{(q^f-1)/p}\equiv 1\mod  q.
\end{eqnarray*}
and 
$\mk q_K$ splits totally  in the extension 
$K \big(\sqrt[p]{<1+\zeta^j>_{j=0,1,\dots,p-1}}\, \big)/K$.

\begin{proof} $ $ Here $\varepsilon_k=1+\zeta^k$, $M=K$ and $L=\Q$.
 The proof is very similar to the case $n=p$ corollary  \ref{c2} starting
here from 
 the  relation
$$u+\zeta^j v\equiv u(1+\zeta^j)\mod  q,$$
for all $j\not\equiv 0\mod  p$
(instead of a congruence  $\mod \  \mk q_K$),  
 observing that 
the degree of $q\mod  p$ can be here greater than $1$.
We have 
$N_{K/\Q}(1+\zeta)=1$
which implies that $\Big(\frac{u}{\mk q_K}\Big)_{\!\! K}=1$
 and  then  $\frac{u^p+v^p}{u+v}=w_1^p\equiv \frac{2^pu^p}{2u}\mod q$ 
implies $\Big(\frac{2}{\mk q_K}\Big)_{\!\!K}=1$.
\end{proof}
\end{cor}
Observe that if $u=x, v=y$ corresponds to a solution of the Fermat's
equation $x^p+y^p+z^p=0,\ p=y$,
from Barlow-Abel relations we get $z+y=x_0^p$, $x+z=p^{\nu p-1}y_0^p$ and
so $x-y=p^{\nu p-1}y_0^p-x_0^p$
which improves the previous corollary with  $\Big(\frac{p}{\mk q_K}\Big)_{\!\! K}=1$ for 
the primes $q|x-y$.
\begin{rema} \label{r1011113}
{\rm The  application of corollaries \ref{c2} and   \ref{c3}  
 implies   that all the $p$-principal primes $q$ dividing $u^p-v^p$   
 verify $q^f\equiv 1 \mod  p^2$, which   brings a new generalization of
 the second Furtw\"angler theorem in   the $SFLT2$ context 
obtained for the  primes $q$ dividing $u-v$  in [GQ] cor. 3.}
\end{rema}

\subsubsection{The cases $n=2p$, $p>3$, and $n=2$}
The reunion of these two   corollaries of the theorem \ref{t1}
allows  us to investigate all the  $p$-principal primes $q$ 
dividing $u^p+v^p$ at the {\it core} of the $SFLT2$ equation.
We need to modify slightly the method to take into account the only values
$k$ with 
$\mk q_K$  co-prime with $u+\zeta^ k v$.

\begin{cor}\label{c4} 
Case $n=2p$ :  suppose that $SFLT2$ fails for  $(p,u,v)$. 
If $q$ is    a  $p$-principal prime  dividing $\frac{u^p+v^p}{u+v}$ 
and $\mk q_K$   is  the prime  ideal $\mk q_K:= (q, u\zeta+v)$ of $\Z_K$,
then
\begin{eqnarray*}
&& q\equiv 1\mod  p^2,\\
&& \Big(\frac{u}{\mk q_K}\Big)_{\!\! K}=\Big(\frac{v}{\mk q_K}\Big)_{\!\!
K}= 
\Big(\frac{p}{\mk q_K}\Big)_{\!\! K}  
= \Big(\frac{1-\zeta^j}{\mk q_K}\Big)_{\!\! K}^{-1}, \ \ \mbox{for\ }
j\not\equiv 0\mod  p,
\end{eqnarray*}
and  
$\mk q_K$ splits totally in 
$K \big(\sqrt[p]{<(1-\zeta^j)/(1-\zeta)>_{j\not\equiv 0\mod  p}}\,
\big)/K$.

\begin{proof}
Here,  we have $M=K=L$  and $\xi = -\zeta$ which implies that  
$v\equiv -\zeta u\mod \mk q$,  thus
$$s_k(u+\zeta v)= u+\zeta^{k} v
=s_k(\gamma)\equiv u(1-\zeta^{k+1})\mod \mk q,\ \,  k=1,\dots p-2.$$
We obtain
$\Big(\frac{u}{\mk q_K}\Big)_{\!\! K}\Big(\frac{1-\zeta^{k+1}}{\mk
q_K}\Big)_{\!\! K}=1$,
 for  $k\not\equiv  p-1\mod  p$,
therefore $\Big(\frac{1-\zeta^2}{\mk q_K}\Big)_{\!\! K}= 
\Big(\frac{1-\zeta^{p-2}}{\mk q_K}\Big)_{\!\! K}$,
so $\Big(\frac{\zeta}{\mk q_K}\Big)_{\!\! K}=1$ and $q\equiv 1\mod  p^2$
which implies that  
$\Big(\frac{1-\zeta}{\mk q_K}\Big)_{\!\! K} 
=\Big(\frac{1-\zeta^{p-1}}{\mk q_K}\Big)_{\!\! K}$. 
Gathering these results, we get
$$\Big(\frac{1-\zeta}{\mk q_K}\Big)_{\!\! K}=\dots 
=\Big(\frac{1-\zeta^{p-1}}{\mk q_K}\Big)_{\!\! K},$$
by multiplication we get 
$\Big(\frac{u^{p-1}}{\mk q_K}\Big)_{\!\! K}\Big(\frac{p}{\mk 
q_K}\Big)_{\!\! K}=1$ 
and finally:
$$\Big(\frac{u}{\mk q_K}\Big)_{\!\! K}= \Big(\frac{v}{\mk q_K}\Big)_{\!\!
K}=
\Big(\frac{p}{\mk q_K}\Big)_{\!\! K}
= \Big(\frac{1-\zeta^j}{\mk q_K}\Big)_{\!\! K}^{\!\!-1}, \ \mbox {\ for 
all\ }  j\not\equiv 0\mod  p. $$ which achieves the proof.
\end{proof}
\end{cor}

\begin{rema}\label{r1011251}
{\rm Suppose moreover that $FLT2$ fails for  $(p,x,y,z)$ with $p|y$.
Then, applying corollary \ref{c4} with $(u,v)=(x,y)$, we get
\be\label{e1111091}
\Big(\frac{1-\zeta}{\mk q_K}\Big)_{\!\! K}^{\!\!-1}
=\Big(\frac{p}{\mk q_K}\Big)_{\!\!K}.
\ee
Moreover $x+y=z_0^p$, so 
$x+y\equiv x(1-\zeta)\equiv z_0^p\mod \mk q_K$.
In an other hand, $\mk q_K|z$, so
$x(1-\zeta)\equiv (x+z)(1-\zeta)=p^{p\nu-1}y_1^p(1-\zeta)\equiv z_0^p\mod 
\mk q_K$ with $\nu\geq 1$, so
$\Big(\frac{1-\zeta}{\mk q_K}\Big)_{\!\!K}=\Big(\frac{p}{\mk 
q_K}\Big)_{\!\!K}$,
so, from (\ref{e1111091}),  $\Big(\frac{p}{\mk q_K}\Big)_{\!\!K}=1$.
We have proved:\footnote{Observe that this  result $\big(\frac{p}{\mk
q_K}\big)_{ K}=1$ 
depending  only on $p$ and $q$  completes strongly 
the relation $q\equiv 1\mod  p^2$  between $p$ and $q$.}}
\end{rema}

\begin{cor}\label{c5}
{\it 
Suppose that $FLT2$ fails  for  $(p,x,y,z)$ with $p\ |\ y$.
If  $q$ is  a  $p$-principal prime dividing $\frac{x^p+y^p}{x+y}$ 
and $\mk q_K$ is  the prime $\mk q_K:= (q, x\zeta+y)$ 
then 
\begin{eqnarray*}
&& q\equiv 1\mod  p^2,\\
&& \Big(\frac{x}{\mk q_K}\Big)_{\!\! K}=\Big(\frac{y}{\mk q_K}\Big)_{\!\!
K}= 
\Big(\frac{p}{\mk q_K}\Big)_{\!\! K} = \Big(\frac{1-\zeta^j}{\mk 
q_K}\Big)_{\!\! K}=1, \, j=1,\dots,p-1, 
\end{eqnarray*}
and  
$\mk q_K$ splits totally in the extension 
$K \big(\sqrt[p]{<1-\zeta^j>_{j\not\equiv 0\mod p}}\, \big)/K$. }
\end{cor}

\begin{rema}
{\rm See theorem \ref{t1107231}  for a  generalization  of this result
without assumption 
that $q$ is a $p$-principal ideal when $\Big(\frac{p}{\mk q_K}\Big)_{\!\!
K}=1$.}
\end{rema}

\begin{cor}\label{c6}
Case $n=2$:  
suppose that $SFLT2$ fails for  $(p,u,v)$.
If  $q$ is  a  $p$-principal prime  dividing $u+v$ of degree $f\mod  p$
and 
 $\mk q_K$   is  any    prime  ideal of $\Z_K$  over  $q$ then 
\begin{eqnarray*}
&& q^f\equiv 1\mod  p^2,\\
&&\Big(\frac{u}{\mk q_K}\Big)_{\!\! K}= \Big(\frac{v}{\mk q_K}\Big)_{\!\!
K}=
\Big(\frac{p}{\mk q_K}\Big)_{\!\! K}
= \Big(\frac{1-\zeta^j}{\mk q_K}\Big)_{\!\! K}^{-1} \ \mbox {\ for all\ }
j\not\equiv 0\mod  p.
\end{eqnarray*}
and  
$\mk q_K$ splits totally in 
$K \big(\sqrt[p]{<(1-\zeta^j)/(1-\zeta)>_{j\not\equiv 0\mod p}}\,
\big)/K$.

\begin{proof}
  Here, $\varepsilon_j=1-\zeta^j$ for $j\not\equiv 0\mod  p$, $M=K$ and
$L=\Q$.
In that case,  we get 
$\Big(\frac{u}{\mk q_K}\Big)_{\!\! K}\Big(\frac{1-\zeta^j}{\mk q_K}\Big)_{\!\! K}=1 \ \mbox{\ for all\ }  j\not\equiv 0\mod  p $.
The end of the proof is similar to that of corollary  \ref{c4}.
\end{proof}
\end{cor}

Observe that if $FLT2$ fails for $p$ and $q|x+y$ (case $n=2$) we get 
$\Big(\frac{p}{\mk q_K}\Big)_{\!\! K}=1$ for the same reason than 
the case $q|x-y$ ($n=1$).

\begin{rema} \label{r1011113}
{\rm The  application of corollaries  \ref{c4} and  \ref{c6}  
 implies   that all the $p$-principal primes $q$ dividing $u^p+v^p$   
 verify $q^f\equiv 1 \mod  p^2$, which   brings a new generalization of
 the first  Furtw\"angler theorem in   the $SFLT2$ 
context  obtained for the  primes $q$ dividing $u+v$ in [GQ] cor. 2.}
\end{rema}

\subsection{The  case of non  $p$-principal  primes $q$}\label{snp}$ $

Assume that $SFLT2$ fails for $(p,u,v)$ with $p|v$.
In this subsection,  we will set  
a  weak  conjecture (so with a proof we can hope  easier) implying that
$SFLT2$ 
equation could  always be  reduced  in the  form 
$$u+\zeta v\in K^{\times p},$$
as soon as $p$ is irregular, which   is  assumed in this subsection.
Let $q\not =p$ be a prime number  and
${\mk q_K}$ any prime ideal of $\Z_K$ above $q$.

\begin{lem}\label{l1108201}
 If   $q$ divides $u$ (resp $v$)  then    $\Big(\frac{v}{\mk q_K}\Big)_{\!\!K}=1$  (resp.  $\Big(\frac{u}{\mk q_K}\Big)_{\!\!K}=1$).
\end{lem}

\begin{proof}
We have $N_{K/\Q} (u+ v\zeta )=\frac{u^p+v^p}{u+v}=w_1^p$, 
where $N_{K/\Q} ({\mk w}_1)=w_1$;
so $q|v$ implies that
$u^{p-1}\equiv w_1^p\pmod q$ which leads to 
$\Big(\frac{u}{\mk q_K}\Big)_{\!\!K}=1$. Similar proof starting from
$q|u$.
\end{proof}

Let $\mathcal S$ be  the   finite set of  smallest non $p$-principal
primes $q$  
such that the set of $p$-classes $c\ell(\mk q_K)\in C\ell$
 of  the prime ideals $\mk q_K$ of $K$ over  $q$  
generates the  $p$-elementary $p$-class group 
$C\ell{[p]}$ of $K$. Let  us note $ Q_p$
 the greatest  prime
$q\in\mathcal S$. 
Let the Minkowski  Bound of $K$ given by 
$$\mathcal
B_p:=(\frac{4}{\pi})^{(p-1)/2}\frac{(p-1)!}{(p-1)^{p-1}}\sqrt{p^{p-2}}.$$
With these definitions we get $Q_p\leq \mathcal B_p$. 

Observe that  under the 
General Riemann Hypothesis GRH, we  know  that the whole ideal class group
of $K$  is generated by the  set of prime 
ideals ${\mk l}$ with
\be\label{e1106091}
N_{K/\Q}({\mk l}) < B:= 12 ({\rm log}\ \Delta_K)^2,
\ee
where $\Delta_K$ is the absolute discriminant of $K$
(see   [BDF]). Under GRH, we have generally  $Q_p\ll\mathcal
B_p$ (where $\ll$ means {\it very small compare to}) as soon as $p$ is large.

\medskip

\begin{lem}\label{l1106092}
Suppose that $u+\zeta v \notin K^{\times p}$.
Then there exists at least one  prime
$q \in \mathcal S$  such that  $uv\not\equiv 0\pmod q$.
\end{lem}

\begin{proof}
Let $\gamma$  be a $p$th root of $u+\zeta v$,     
$\gamma:=\sqrt[p]{u+\zeta v}$.
Let $H_1$ be the $p$-elementary Hilbert class field of $K$ (so that
${\rm Gal} (H_1/K) \simeq C\ell_{[p]}$). Let $N_1$ be a subextension of
$H_1$ such that $H_1$ is the direct compositum of $N_1$
and $K(\sqrt [p] \gamma)$ over $K$.

Therefore there exists at least one   prime $q\in \mathcal S$ such that  
the Frobenius of all  the prime ideals ${\mk q}_K$ over $q$ in $H_1/K$ are  of order $p$ and fix
$N_1$, so that their  restriction to $K(\sqrt [p] \gamma)/K$ are  of order $p$.
Thus
$\Big(\frac{u+\zeta v}{{\mk q}_K} \Big)_{\!\!K}\not =1$.

(i) If $q | v$,  we get a  contradiction with  lemma \ref{l1108201},
so $v\not\equiv 0\mod q$.

(ii) If $q | u$ we have $\Big(\frac{u+\zeta v}{{\mk
q}_K} \Big)_{\!\!K} =\Big(\frac{\zeta v}{{\mk
q}_K} \Big)_{\!\!K}=\Big(\frac{\zeta }{{\mk q}_K }\Big)_{\!\!K}$
because $\Big(\frac{ v}{{\mk q}_K} \Big)_{\!\!K}=1$ from  Lemma
\ref{l1108201}
and thus  $\Big(\frac{u+\zeta v}{{\mk q}_K} \Big)_{\!\!K}=1$
since $\kappa\equiv 0\pmod p$ from the first Furtwangler's theorem for
SFLT (see
[GQ, Corollary 2.10]), giving also a contradiction with
$\Big(\frac{u+\zeta
v}{{\mk q}_K} \Big)_{\!\!K}\not =1$.Therefore $u\not\equiv 0\pmod q$.
\end{proof}

\begin{defi}
{\rm 

\smallskip

For a definition of the  character of Teichm\"uller $\omega$ 
of $Gal(K/\Q)$ see [GQ] section 1.3. 
Let us consider the characters  $\chi_i=\omega^i, \ 1\leq i\leq p-1$. 
Let $\mathcal E$ be the group of $p$-primary pseudo-units  of $K$ seen as a $\F_p[g]$-module,  
and the $\chi_i$-components $\ml E_i:=\mathcal E^{e_{\chi_{i}}}$  of $\mathcal E$ interpreted in  the group $\ml E/\ml E^p$. 
The components  $\mathcal E_i$ are not all
trivial because $p$ is irregular.
}
\end{defi}

\begin{thm}\label{t1012211}
Suppose that $SFLT2$ fails for  $(p,u,v)$   with  $u+\zeta v\not\in
K^{\times p}$.  Then
$p$ is irregular and there exists at least 
one non $p$-principal prime $q\in \mathcal S$   such that: 
\bn
\item
$q\not| uv $ 

\item 
Let $n$ be  the order   of $\frac{v}{u}\mod q$,  $\xi:=e^{\frac{2\pi
i}{n}}$ and 
$\mk q:= (u\xi- v, q)$   prime ideal of $\Z[\xi]$ above $q$. 
There exists  at least one  integer $m\not\equiv 0\mod p$ such that 
all the prime ideals   $\mk Q$ of $\Z[\xi,\zeta]$ dividing $\mk q$  split
totally
in   the Kummer extension 
$$M \big(\ {\sqrt[p]{< ((1+\xi\zeta^k)\zeta^{-k^{m}})
/((1+\xi\zeta)\zeta^{-1})>_{k=1,\ldots,p-2}} }\  \big) \big/M.$$
\end{enumerate}
\end{thm}

\begin{proof}$ $
\bn
\item

$p$ is irregular from [Gr2] thm 2.2.
From lemma \ref{l1106092} it is possible to choose 
$q\in \mathcal S$  with $uv\not\equiv 0\mod p$. 
\smallskip 
\item
The pseudo-unit $\gamma=\sqrt[p]{u+\zeta v}$  is not a $p$th power, hence in the decomposition
$\gamma=\prod_{{\chi_i}}\gamma^{e_{\chi_i}}$
on the p-1 characters $\chi_i,\  i=1,\dots,p-1$, there exists at least one  $i=m$ such that the pseudo-unit  $\gamma^{e_{\chi_m}}$ be  not a $p$-power.
Let us name it  $\gamma_m$.

$\gamma_m$ is a $p$-primary pseudo-unit;
from Hilbert's class field theory, decomposition and reflection  theorems and lemma \ref{l1106092} applied with $H_1$ and $\gamma_m$, it is possible to choose  one $\mk q_K\in \ml S $  
such that 
$$\Big(\frac{\gamma_m}{\mk q_{K}}\Big)_{\!\!K}=\zeta^{w_m}\mbox{\ with\ } w_m\not\equiv 0\mod p
\mbox{\ and\ }\Big(\frac{\gamma_i}{\mk q_{K}}\Big)_{\!\!K}=1 \mbox{\ for all \ } i\not=m.$$
\item
Here the extension $M(\sqrt[p]{\gamma_m})$ is  Galois on $\Q$  because its  Galois group 
 acts in letting  globally  unvarying the radical, when raising to a power prime to $p$  by use of the idempotent.
We can always change $\mk q_K$ in acting  by  conjugation to obtain $w_m=1$ and so 
$$\Big(\frac{\gamma}{\mk q_K}\Big)_{\!\!K}= \Big(\frac{\gamma_m}{\mk q_K}\Big)_{\!\!K}=\zeta.$$ 
\item
From $s_k(\gamma)=u+\zeta^k v$ for $k=1,\dots,p-1$ and $u\xi-v\equiv
0\mod\mk q$ we get $$s_k(\gamma)\equiv u\varepsilon_k\mod \mk q,$$ 
so 
$$\Big(\frac{s_k(\gamma)}{\mk q_K}\Big)_{\!\!K}
=\Big(\frac{u\varepsilon_k}{\mk Q}\Big)_{\!\!M}=\Big(\frac{s_k(\gamma_m)}{\mk q_K}\Big)_{\!\!K}=\zeta^{ k^m},$$
so 
$$\Big(\frac{u(1+\xi\zeta^k)}{\mk Q}\Big)_{\!\!M}=\zeta^{k^m},$$
and also  
$$\Big(\frac{u(1+\xi\zeta)}{\mk Q}\Big)_{\!\!M}=\zeta,$$
which leads to the result.
\en
\end{proof}

This theorem leads us to set the following {\it criterion}
\footnote{We use intentionally    the term {\it criterion} to indicate
that the corollary \ref{c1012211} allows us (at least theoretically) in a
finite number of arithmetic computations
to  prove that for $p$ given  the SFLT2 equation can be reduced to the
form $u+\zeta v\in K^{\times p}$.} 
for the SFLT2 equation to take the reduced form $u+\zeta v\in K^{\times
p}$ for the irregular prime  $p>3$:
\begin{cor}\label{c1012211}
Let $p>3$ be an odd  irregular prime. Assume that SFLT2 fails for $p$ and
that   for  all the    primes $q\in\mathcal  S$  and all  the integers
$n>2$ dividing $q-1$, 
 there is  no  prime ideal $\mk q$ of $\Z[\xi]$ above $q$ which  splits
totally   in the Kummer extension
$$\Q(\xi,\zeta) 
\big(\sqrt[p]{<((1+\xi\zeta^{k})\zeta^{- k^m})/((1+\xi\zeta)\zeta^{-})>_{k=1,\dots,p-2}}\, 
\big)\big/\Q(\xi,\zeta)$$
with   $\xi:=e^{\frac{2\pi i}{n}}$, 
and   $m$  an integer $\not \equiv 0\mod p$.
Then  the solution(s) of the $SFLT2$  equation take(s) the reduced form
$u+\zeta v\in K^{\times p}$.
\end{cor}

We conjecture that the the corollary \ref{c1012211} is true for all the
irregular primes $p$. In the other hand,   we know that $u+\zeta v\in
K^{\times p}$ when SFLT2 fails for $p$ regular, which leads us  to set the
conjecture:

\begin{conj}\label{cj1012211}
Let $p$ be an odd prime. If SFLT2 conjecture  fails for $p$  then 
the solution(s) of the $SFLT2$  equation take(s) the reduced form $u+\zeta
v\in K^{\times p}$.
\end{conj}

\begin{rema}
{\rm 
Note that the following probability estimates can  in no way  be
considered as an element of  a proof of conjecture \ref{cj1012211}.
Suppose, as an example,  that $SFLT2$ fails for  $p$ 
and that  $p\|\mathcal C\ell_ K$  class group  of $K$, which implies that
$Card(\ml S)=1$.
Under GRH, from [BDF],  the definition of $\mathcal S$ should imply  
that $Q_p< 12(p-2)^2 (log\ p)^2$.
For one $q\in\mathcal S$, the probability estimate  $\mathcal P$ that $\mk
q$ splits totally in  the Kummer extension
$$\Q(\xi,\zeta) 
\big(\sqrt[p]{<((1+\xi\zeta^k)\zeta^{- k^m})/((1+\xi\zeta)\zeta^{-1})>_{k=1,\dots,p-2}}\, 
\big)\big/\Q(\xi,\zeta)$$
of degree $p^\delta$ could be  very  roughly $\mathcal  P< \mathcal
O(\frac{ Q_p}{p^{\delta}})$ because $\phi(q)< Q_p$ for
$q\in\mathcal S$.
The probability $\mathcal P^\prime$ for  the conjecture \ref{cj1012211} be
true for  $p$ could verify   $$1>\mathcal P^\prime >1-\mathcal
O(\frac{ Q_p}{p^{\delta}}).$$
Note that often,  and perhaps for all irregular primes $p>10^3$, we have
$\delta>\frac{p}{4}$ and, under GRH we have   $Q_p<p^3$, so {\it roughly }
$$1>\mathcal P^\prime >1- \frac{1}{p^{\frac{p}{4}-3}}.$$
}
\end{rema}

\begin{rema}$ $
{\rm
\bn
\item
The reduced form $u+\zeta v=\gamma^p$ can be an important tool 
to tackle the $SFLT2$ conjecture on a diophantine approach: 
we have $u+\zeta^k v= s_k(\gamma)^p$ for $k\not\equiv 0\mod p$, which
implies that  
\begin{displaymath}
\begin{split}
& (\zeta-\zeta^2)v=\gamma^p-s_2(\gamma^p),\\
& (\zeta^2-\zeta^3)v=s_2(\gamma^p)-s_3(\gamma^p),\\
\end{split}
\end{displaymath}
which brings, among a lot of possible strong diophantine equations,  the
{\it Fermat's type} equation
$$\zeta\gamma^p+s_3(\gamma^p)-(1+\zeta)s_2(\gamma^p)=0,$$
where $s_k: \zeta\rightarrow \zeta^k$ is a $\Q$-isomorphism of $K$,
form of equation which could be used to try to tackle  SFLT2  conjecture, 
see for instance Washington [Was] chap 9   or Ribenboim  [Rib1]  chap. 3
or Cohen [Coh] 6.9.5 with  
some diophantine  approachs of the second case.

For instance, we can easily find again in this case a result we will  also obtain   
in a broader context (theorem \ref{t1107232}):  {\it $$\Big(\frac{1+\zeta^k}{\mk q_K}\Big)_{\!\!K}=1\mbox{\ for all\ }k=1,\dots,p-1,$$
where $\mk q_K$ is any prime ideal dividing $\gamma^p=u+v\zeta$.}
%
%
%
%
%
%
%
%
%
%
%
%
%
%
%
%
%
%
\item More generally we can prove that: 
{\it  If SFLT2 failed  for $(p,u,v)$ with $u+\zeta v=\gamma^p \mbox{\ and\
}p\ |\ v,$ 
then we should have  
$$(s_k(\gamma)+\zeta_{p^2}^{k-1} s_{p-k+2}(\gamma))\Z_{K_{p^2}}
=\mk W_1^p\mbox{\ for\ } k=2,3,\dots,p-4,$$
where $s_k(\gamma)$ are $p$-primary pseudo-units,  $\zeta_{p^2}$ is the $p^2$th root of unity 
such that $\zeta_{p^2}^p=\zeta$, $K_{p^2}=\Q(\zeta_{p^2})$, 
and $\mk W_1$ is an ideal of $\Z_{K_{p^2}}$.}
\item
 A strategy of proof of SFLT2 conjecture  for the prime $p$ could then  be
driven in two steps:
\bn
\item
Reduce the SFLT2 equation to the form $u+\zeta v\in K^{\times p}$
implicitly if $p$ is regular and 
by proving that the criterion of corollary \ref{c1012211} is verified for
$p$ if $p$ is irregular.
\item
Prove that the diophantine equation $u+\zeta v\in K^{\times p}$ with $p|v$
has no solution
by a diophantine different approach.
\en
\en
}
\end{rema}

\section {On the second case of Fermat's Last Theorem}\label{s1011101}
\label{FLT2}
This section details some results obtained for   the SFLT conjecture in
the second case FLT2 
of FLT theorem.  
The last subsection  focus more particularly on some strong properties of
the primes
dividing $\frac{(x^p+y^p)(y^p+z^p)(z^p+x^p)}{(x+y)(y+z)(z+x)}$ 
or  $\frac{(x^p-y^p)(y^p-z^p)(z^p-x^p)}{(x-y)(y-z)(z-x)}$
if FLT2 failed for $(p,x,y,z)$ with $p|y$.
\subsection{A  conjecture for  the second case of FLT} 
This subsection  deals with  an application of the previous results
obtained for SFLT2  to the 
FLT2 context.
We give a weak conjecture
\footnote{To avoid any misunderstanding, we mean by {\it  weak conjecture}, a conjecture 
with weak assumptions that we could hope more easily reachable.}
      which should imply
the generalization to FLT2 of Terjanian's  theorem for FLT1.

We call here  {\it  weak FLT2} theorem the assertion:\footnote{This
theorem is 
in the continuity  for the second case $FLT2$ of the Terjanian theorem
for the first case $FLT1$ [Ter] and Ribenboim [Rib1, 6C, p. 18].}
  
{\it Let $p$ be an odd prime. There are  no  coprime integers
$x,y,z\in\Z$, such  that 
\be\label{e1}
x^p+y^p+z^p=0, \ p|y,
\ee
with $x$ and $z$  square integers.}

Let $x=\alpha^2,\ z=\beta^2, \alpha,\beta\in\Z$.
Observe that  the proof of this theorem is immediate as soon as $p\equiv
3\mod  4$,
because $\alpha$ and $\beta$ are coprime and 
$x+z=\alpha^2+\beta^2\equiv 0\mod  p$,
contradiction. 

\begin{lem}\label{l3}
Suppose that $p$ is a prime with $p\equiv 1\mod  4$ and that  $x,y$ 
verify the Fermat's equation  
$$ x^p+y^p+z^p=0,\ p|y,\ x, y\mbox{\ square integers\ }.$$
Let $q$ be a prime with $q\equiv 3\mod  4$ and $p$ coprime to 
$\kappa=\frac{q^f-1}{p}$ where $f$ is the order of $q\mod p$. 
Then $q\not| \ xy(x^2-y^2)$.

\begin{proof}$ $
Suppose that  $q | xy(x^2-y^2)$.
 We have assumed that $p| y$, thus 
$x^2-y^2\not\equiv 0\mod p$. We have  $p$ prime to $\kappa$,  thus  $
x(x+y)\not\equiv 0\mod q$  
from the first  theorem of Furtwangler 
and  $x-y\not\equiv 0\mod q$ from the second theorem of Furtw\"angler, so
$y\equiv 0\mod q$.
  From Barlow-Abel relations 
$$x+z =p^{\nu p-1} y_0^p,\ \frac{x^p+z^p}{x+z}= py_1^p,\ y=- p^\nu 
y_0y_1, \ \nu\geq 1,$$
If $y\equiv 0\mod q$, then we cannot assert that $q^{p-1}-1\equiv 1\bmod
p^2$, 
because  Furtw\"angler theorem cannot be applied here.
Nevertheless we will show  the reduced result:

{\it  if  $q\not=p$ verifies $y\equiv 0\mod q$ 
and $ x+z\not\equiv 0\mod q$ then $q-1\equiv 0\mod  p^2$}:

Suppose that  $q|\frac{x^p+z^p}{x+z}$ with $p$ prime to $\kappa$  and
search for a contradiction:
let $\mk q_K$  be a prime ideal of $\Z_K$ lying over $q$. From $q|y$ 
and the Barlow-Abel relation  $x+y=z_0^p$, we have
so 
$$\Big(\frac{x}{\mk  q_K}\Big)_{\!\! K}=\Big(\frac{x+y}{\mk
q_K}\Big)_{\!\! K}=
\Big(\frac{z_0^p}{\mk  q_K}\Big)_{\!\! K}=1.$$ 
Similarly $\Big(\frac{z}{\mk q_K}\Big)_{\!\!K}=1$, 
so $x^{(q-1)/p} - z^{(q-1)/p}\equiv 0\mod  \mk q_K$. 
We get $$q\ |\  x^{(q-1)/p}-z^{(q-1)/p} \mbox{\  and \ } 
q\ |\ x^p+z^p.$$  
If we suppose $\kappa=\frac{q-1}{p}$ prime 
to $ p$, we have $\kappa=\frac{q-1}{p}$ even  
and $ x^\kappa \equiv (-z)^\kappa \mod  q$  and $ x^p \equiv  (-z)^p
\mod   q$,  
thus  $q \ |\  x+z$  by a  B\'ezout relation between  $p$ and $n$
(absurd). 

\smallskip
Therefore, all  the prime  factors  $r$  of  $y_1$ (where $y=-y_0y_1$) verify  $r\equiv
1\mod  p^2$, 
thus $q| y_0$ because by assumption  $p\|q-1$. 
Therefore $$q\ |\ y_0\ |\  x+z= \alpha^2+\beta^2,$$
contradiction because $q\equiv 3\mod  4$ and $\alpha, \beta$ are coprime, 
therefore $q\not|\ xy(x^2-y^2)$.
\end{proof}
\end{lem}

From $ xy(x^2-y^2)\not\equiv 0\mod q$ and from theorem \ref{t1}, we are
led 
to set the following conjecture  independent of the Fermat context.

\begin{conj}\label{cj3}
{\it ``Let $p>3$ be a prime. There exists at least one  
$p$-principal prime $q\equiv 3\mod 4$ with $p$ prime to
$\kappa=\frac{q^f-1}{p}$  such that, 
for  all the integers  $n >2$ dividing $q-1$, there is no  prime ideal 
$\mk q$ of $\Z[\xi]$ above $q$ which  splits totally   in the Kummer
extension
$$\Q(\xi,\zeta) 
\big(\sqrt[p]{<(1+\xi\zeta^{k})/(1+\xi\zeta)>_{k=1,\dots,p-2}}\, 
\big)\big/\Q(\xi,\zeta)$$
where $\xi:= e^{\frac{2\pi i}{n}}$."} 
\end{conj}

\begin{thm}\label{t2}
If the conjecture \ref{cj3} is true then the weak $FLT2$ theorem is true.%
\footnote{This theorem and the Terjanian theorem imply that 
if the conjecture \ref{cj3} is true then the Fermat equation 
$x^{2p}+y^{2p}+z^{2p}=0$
has no solutions with $x,y,z$ nonzero integers.}

\begin{proof}$ $

Suppose that conjecture \ref{cj3} is true and  weak $FLT2$ theorem is
false and search for a contradiction:

\smallskip

(i) If $p\equiv 3\mod  4$:  it results directly  
from $x+z=\alpha^2+\beta^2\equiv 0\mod  p$ impossible, contradiction.

\smallskip

(ii) If $p\equiv 1\mod  4$:  there exists 
at least one prime $q\equiv 3\mod  p$ with $p\|q^{p-1}-1$ 
verifying the  conjecture \ref{cj3}. 
Then  $ xy(x^2-y^2)\not\equiv 0\mod p$ from lemma \ref{l3},
so we can apply theorem \ref{t1} to this prime $q$:  
the decomposition of   $q$ in a Kummer extension given in the theorem
\ref{t1} 
contradicts the decomposition of $q$ given in the conjecture \ref{cj3},
for the same extension, contradiction which achieves the proof.

\end{proof}
\end{thm}

\begin{rema} \label{r8}$ $
{\rm 
\bn
\item
The conjecture \ref{cj3} independent of Fermat is highly probable,
considering,
with a probabilistic approach, that  it suffices to take the smallest 
$p$-principal primes  $q$ with $p\|q^{p-1}-1$ and $q\equiv 3\mod  4$
 to find easily at least one of them with conjecture \ref{cj3} verified.
 Suppose $p>125000$ (we know in the classical approach that $FLT2$ is true
for $p<125000$,
 see Ribenboim [Rib1, p. 199]). In practice we will find, with these
assumptions on $q$,
 some primes $q\ll p$  with a {\it very large} probability 
that the conjecture \ref{cj3} is true for at least  one of these primes
$q$.
\item  As an example, we suggest that a very rough estimate of the
probability $\mathcal P(p)$
that the conjecture \ref{cj3}  be true at $p$ (so implying  weak $FLT2$
theorem at $p$)
for the first prime $p$ greater  than $125000$) verify:
$$1-\Big(\prod_{\substack{q\,{\rm \ prime\ }\, p{\rm-principal} \\ q<\
20000,\, q\equiv 
3 \mod 4 \\  p\|q^{p-1}-1}}
\frac{1}{q^{30000}}\ \Big) \ \  <\ \mathcal P(p)\ \ \leq 1.$$ 
\en
Note that this probability estimate   can  in no way  be considered as an
element of  a proof of  the weak FLT2 theorem.
}
\end{rema}

\subsection{Some properties of  the  primes $q$ 
dividing $\frac{(x^p-y^p)(y^p-z^p)(z^p-x^p)}{(x-y)(y-z)(z-x)}$ }

We assume that the second case  $FLT2$ fails for $(p,x,y,z)$ with $p|y$.
In this subsection, we give,  for a possible future use, 
some general  strong  properties  of decomposition of  the primes $q$
dividing  
$\frac{(x^p-y^p)(y^p-z^p)(z^p-x^p)}{(x-y)(y-z)(z-x)}$ in certain
$p$-Kummer extensions.
Here, we don't assume  that $q$ is $p$-principal or not, 
thus this subsection brings complementary informations to corollary
\ref{c2}. 
The Furtwangler's theorem cannot be used for these primes $q$, so  we
don't know if $p^2|q-1$ or not.

Let us define  here the  totally real cyclotomic units   
$$\epsilon_{a}=:\zeta^{(1-a)/2}\cdot\frac{1+\zeta^{a}}{1+\zeta},\ 1 \leq
a\leq p-1,$$
where we observe that $\epsilon_1=1$ with this definition.
\footnote{Be careful, $\epsilon_a$ (epsilon in Latex) is different of  the Vandiver
cyclotomic unit  
$\varepsilon_{a}=: 1+\xi\zeta^a$  (varepsilon in Latex) defined in the definition \ref{d11O9221}.}
Let us note that
\be\label{e1111103} 
\varepsilon_{p-a}=\zeta^{(1-(p-a))/2}\cdot\frac{1+\zeta^{p-a}}{1+\zeta}= \zeta^{(1+a)/2}\cdot\frac{1+\zeta^{-a}}{1+\zeta}= \zeta^{(1-a)/2}\frac{1+\zeta^a}{1+\zeta}=\varepsilon_a.
\ee
\begin{lem}\label{l1107011}
Suppose that $FLT2$ fails for $(p,x,y,z)$ with $p|y$ .
Let $q\not= p$ be a prime and  $\mk q_K$ be a  prime ideal of $\Z_K$ over
$q$. Then we have 
for $k=1,\dots,p-1$: 

(i) If $\mk q_K|x\zeta-y$  then $\Big(\frac{x+\zeta^k y}{\mk
q_K}\Big)_{\!\! K}
=\Big(\frac{\zeta^{k/2}}{\mk q_K}\Big)_{\!\!
K}\Big(\frac{\epsilon_{k+1}}{\mk q_K}\Big)_{\!\! K}$. 

(ii) If $ \mk q_K|z\zeta-y$  then $\Big(\frac{z+\zeta^k y}{\mk
q_K}\Big)_{\!\! K}
=\Big(\frac{\zeta^{k/2}}{\mk q_K}\Big)_{\!\!
K}\Big(\frac{\epsilon_{k+1}}{\mk q_K}\Big)_{\!\! K}$.

(iii) If $\mk q_K|x\zeta-z$    
then $\Big(\frac{x+\zeta^k z}{\mk q_K}\Big)_{\!\! K}\Big(\frac{p }{\mk
q_K}\Big)_{\!\! K}
=\Big(\frac{\zeta^{k/2}}{\mk q_K}\Big)_{\!\!
K}\Big(\frac{\epsilon_{k+1}}{\mk q_K}\Big)_{\!\! K}$.
\begin{proof}$ $

(i) 
From  $x\zeta - y\equiv 0\mod\mk  q_K$ we get   
$$x+\zeta^k y\equiv x(1+\zeta^{k+1})\mod \mk q_K,\ k=1,\dots,p-1.$$
thus 
$$\frac{x+\zeta^k y}{x+y}\equiv 
\frac{1+\zeta^{k+1}}{1+\zeta}\mod \mk q_K,\mbox{\ for\ } k=1,\dots,p-1.$$
In the other hand, for $1\leq k\leq p-2$ then  
$\epsilon_{k+1}=\zeta^{(1-(k+1))/2}\cdot\frac{1+\zeta^{k+1}}{1+\zeta}$ 
is a totally real cyclotomic unit, so  
$\frac{x+\zeta^k y}{x+y}\equiv \epsilon_{k+1}\zeta^{k/2}\mod\mk q_K,\
k=1,\dots p-1$, 
and finally 
$$\Big (\frac{x+\zeta^k y}{\mk q_K}\Big)_{\!\! K}
=\Big(\frac{\zeta^{k/2}}{\mk q_K}\Big)_{\!\!
K}\Big(\frac{\epsilon_{k+1}}{\mk q_K}\Big)_{\!\! K} 
\mbox{\ for\ } k=1,\dots,p-2,$$
because  $x+y\in K^{\times p}$.
\footnote{ We don't know here if $p^2|q-1$.}

(ii) the proof is similar with $z$ in place of $x$. 

(iii) In that case we have $x+z=p^{\nu p-1}y_0^p$ with $\nu>0$ and so
$x+z\in p^{-1}K^{\times p}$.
\end{proof}
\end{lem}

\begin{rema}
{\rm This  property of the primes $q$ dividing $\frac{x^p-y^p}{x-y}$ or
$\frac{y^p-z^p}{y-z}$, 
or  $\frac{x^p-z^p}{x-z}$   is strong  because $x+\zeta^k y$ or $y+\zeta^k
z$,
 or $\frac{x+\zeta^k z}{1-\zeta^k})$  and $\epsilon_{k+1}$ are
pseudo-units  
not linearly connected by  an action of $\F_p[g]$.
}
\end{rema}

\begin{thm}\label{t1107232}
Suppose  that   the second case of FLT fails for $(p,x,y,z)$ with $p|y$ .
Let $q$ be a prime dividing $\frac{x^p-y^p}{x-y}$ (or
$\frac{y^p-z^p}{y-z}$).
Let $\mk q_K$ be \underline{the} prime ideal of $\Z_K$ over $q$
dividing $x\zeta- y$ (or $z\zeta-y$). 
Assume that  the $p$-class   $c\ell(\mk q_K)\in C\ell^-$.
\footnote{Note that as soon as Vandiver's conjecture is true for $p$, this assumption is verified.}

(i) If $p^2\not|\ q-1$   then  $q$ is not $p$-principal and

$$\Big(\frac{\epsilon_{p-2k'-1}}{\mk q_K}\Big)_{\!\! K}
=\Big(\frac{\zeta^{-k'(k'+1)}}{\mk q_K}\Big)_{\!\! K}
\mbox{\ for\ }1\leq k'\leq \frac{p-3}{2},$$
and 
$$\Big(\frac{\epsilon_{p-2k'}}{\mk q_K}\Big)_{\!\! K}
=\Big(\frac{\zeta^{\frac{1}{4}-k'^2}}{\mk q_K}\Big)_{\!\! K}
\mbox {\ for\ }1\leq k'\leq \frac{p-3}{2}.$$

(ii) If $p^2|\ q-1$ then 
$$\Big(\frac{1+\zeta^j}{\mk q_K}\Big)_{\!\! K}=1\mbox {\ for\ } j=1,\dots
p-1.$$

\begin{proof}$ $

(i)  Let us  suppose at first that $p^2\not|\ q-1$: We know that $q$ is
not $p$-principal,
if not it should imply $p^2|q-1$ from corollary \ref{c2}.

-  From   previous lemma \ref{l1107011},  we have
\be\label{e1111101} 
\Big (\frac{x+\zeta^k y}{\mk q_K}\Big)_{\!\! K}
=\Big(\frac{\zeta^{k/2}}{\mk q_K}\Big)_{\!\!
K}\Big(\frac{\epsilon_{k+1}}{\mk q_K}\Big)_{\!\! K} 
\mbox{\ for\ } k=1,\dots,p-2,
\ee
and so, with $p-k$ in place of $k$,
\be\label{e1111102} 
\Big (\frac{x+\zeta^{p-k} y}{\mk q_K}\Big)_{\!\! K}
=\Big(\frac{\zeta^{(p-k)/2}}{\mk q_K}\Big)_{\!\!
K}\Big(\frac{\epsilon_{p-k+1}}{\mk q_K}\Big)_{\!\! K} 
\mbox{\ for\ } p-k=1,\dots,p-2.
\ee

-  For $2\leq k\leq p-2$, we can write 
$x+\zeta^k y= A_k B_k\alpha^p$ with $\alpha\in K^{\times p}$,  
pseudo-units $A_k, B_k$ verifying $A_k^{s_{-1}+1}\in K^{\times p}$ 
and $B_k^{s_{-1}-1}\in K^{\times p}$ where we recall that $s_k$ is the
$\Q$-isomorphism $s_k: \zeta\rightarrow \zeta^k$ of $K$. 
Let $\Big(\frac{A_k}{\mk q_K}\Big)_{\!\! K}=\zeta^w$, we get 
$$\Big(\frac{A_k^{s_{-1}}}{s_{-1}(\mk q_K)}\Big)_{\!\! K}
=\Big(\frac{A_k^{-1}}{s_{-1}(\mk q_K)}\Big)_{\!\! K}=\zeta^{-w},$$
so $\Big(\frac{A_k}{s_{-1}(\mk q_K)}\Big)_{\!\! K}=\zeta^{w}$, 
and so $\Big(\frac{A_k}{\mk q_K s_{-1}(\mk q_K)}\Big)_{\!\! K}=\zeta^{2w}$.
But $c\ell(\mk q_K)\in C\ell^-$, so $(\mk q_Ks_{-1}(\mk
q_K))^n\Z_K=\beta\Z_K$ 
with $\beta\in \Z_K$ and $gcd(n,p)=1$. Then 
$$\Big(\frac{A_k}{\mk q_K^n s_{-1}(\mk q_K)^n}\Big)_{\!\! K}=
\Big(\frac{A_k}{\beta}\Big)_{\!\! K} =1,$$ 
because $A_k$ is a $p$-primary pseudo-unit 
(for instance by application of Artin-Hasse reciprocity law), so $w=0$ and 
$\Big(\frac{A_k}{\mk q_K}\Big)_{\!\! K}=1$.

-  We get $\frac{x+\zeta^k y}{x+\zeta^{p-k} y}\in A_k^2\times K^{\times
p}$, so
\be\label{e1111105} 
\Big(\frac{x+\zeta^k y}{\mk q_K}\Big)_{\!\!}
=\Big(\frac{x+\zeta^{p-k}y}{\mk q_K }\Big)_{\!\! K}\mbox{\ for \ }k=2,\dots,p-2,
\ee
which leads  from (\ref{e1111101}) and (\ref{e1111102}) to
\be\label{e1107081}
\Big(\frac{\epsilon_{p-k+1}}{\mk q_K}\Big)_{\!\! K}
=\Big(\frac{\zeta^{k}}{\mk q_K}\Big)_{\!\!
K}\Big(\frac{\epsilon_{k+1}}{\mk q_K}\Big)_{\!\! K} 
\mbox{\ for\ } k=2,\dots,p-2.
\ee

- In the other hand,  from (\ref{e1111103}) we  have    
\be\label{e1107082}
 \epsilon_{p-k-1} = \epsilon_{k+1}:
\ee
From (\ref{e1107081}) and (\ref{e1107082}) we derive that
\begin{equation}\label{e1108211}
\Big(\frac{\epsilon_{p-k-1}}{\mk q_K}\Big)_{\!\! K}
=\Big(\frac{\zeta^{-k}}{\mk q_K}\Big)_{\!\!
K}\Big(\frac{\epsilon_{p-k+1}}{\mk q_K}\Big)_{\!\! K}
\mbox{\ for\ k=2,\dots,p-2}.
\ee

- We get  for the values $k=2k'$ even 
$$\Big(\frac{\epsilon_{p-2k'-1}}{\mk q_K}\Big)_{\!\! K}
=\Big(\frac{\zeta^{-2k'}}{\mk q_K}\Big)_{\!\!
K}\Big(\frac{\epsilon_{p-2k'+1}}{\mk q_K}\Big)_{\!\! K}
\mbox{\ for\ }1\leq k'\leq \frac{p-3}{2}.$$
 Observing that $\epsilon_{p-1}=1$, so 
$\Big(\frac{\epsilon_{p-1}}{\mk q_K}\Big)_{\!\! K}=1$
we get  inductively 
$$\Big(\frac{\epsilon_{p-2k'-1}}{\mk q_K}\Big)_{\!\! K}
=\Big(\frac{\zeta^{-\sum_{j=1}^{k'} 2j}}{\mk q_K}\Big)_{\!\! K}
\Big(\frac{\epsilon_{p-1}}{\mk q_K}\Big)_{\!\! K}
\mbox{\ for\ } k'=1,2,\dots, \frac{p-3}{2},$$
so
$$\Big(\frac{\epsilon_{p-2k'-1}}{\mk q_K}\Big)_{\!\! K}
=\Big(\frac{\zeta^{-k'(k'+1)}}{\mk q_K}\Big)_{\!\! K}
\mbox{\ for\ }0\leq k'\leq \frac{p-3}{2}.$$
 
- We get for the odd values $k=2k'+1$
$$\Big(\frac{\epsilon_{p-(2k'+1)-1)}}{\mk q_K}\Big)_{\!\! K}
=\Big(\frac{\zeta^{-(2k'+1)}}{\mk q_K}\Big)_{\!\! K}
\Big(\frac{\epsilon_{p-(2k'+1)+1}}{\mk q_K}\Big)_{\!\! K}
\mbox{\ for\ }k' = \frac{p-3}{2},\frac{p-5}{2}\dots,1,$$
so 
$$\Big(\frac{\epsilon_{p-2k'}}{\mk q_K}\Big)_{\!\! K}
=\Big(\frac{\zeta^{2k'+1}}{\mk q_K}\Big)_{\!\! K}
\Big(\frac{\epsilon_{p-2k'-2}}{\mk q_K}\Big)_{\!\! K}
\mbox{\ for\ } k' = \frac{p-3}{2},\frac{p-5}{2}\dots,1.$$
Observing that $\epsilon_1=1$, so 
$\Big(\frac{\epsilon_1}{\mk q_K}\Big)_{\!\! K}=1$
we get for $k'=\frac{p-3}{2}$, so $2k'+1=p-2$,
$$\Big(\frac{\epsilon_{3}}{\mk q_K}\Big)_{\!\! K}
=\Big(\frac{\zeta^{-2}}{\mk q_K}\Big)_{\!\! K}
\Big(\frac{\epsilon_{1}}{\mk q_K}\Big)_{\!\! K},$$
and for $k'=\frac{p-5}{2}$
$$\Big(\frac{\epsilon_{5}}{\mk q_K}\Big)_{\!\! K}
=\Big(\frac{\zeta^{-4}}{\mk q_K}\Big)_{\!\! K}
\Big(\frac{\epsilon_{3}}{\mk q_K}\Big)_{\!\! K},$$
and so on.

- Let us define  $k":=\frac{p-1}{2}-k'$, we get 
$$2k'+1=p-2k", \mbox{\ for\ } k'=\frac{p-3}{2},\dots,1\mbox{\
corresponding to \ } k"=1, \dots,\frac{p-3}{2}.$$
It follows that  
$$\Big(\frac{\epsilon_{p-2k'}}{\mk q_K}\Big)_{\!\! K}
=\Big(\frac{\zeta^{\sum_{j=1 }^{k"} -2j}}{\mk q_K}
\Big)_{\!\! K}\Big(\frac{\epsilon_1}{\mk q_K}\Big)_{\!\! K}
\mbox{\ for\ } k'= \frac{p-3}{2},\frac{p-5}{2},\dots, 1,$$
so
$$\Big(\frac{\epsilon_{p-2k'}}{\mk q_K}\Big)_{\!\! K}
=\Big(\frac{\zeta^{-k"(k"+1)}}{\mk q_K}
\Big)_{\!\! K}\mbox {\ for\ }1\leq k'\leq \frac{p-3}{2},$$
so
$$\Big(\frac{\epsilon_{p-2k'}}{\mk q_K}\Big)_{\!\! K}
=\Big(\frac{\zeta^{-(\frac{p-1}{2}-k')(\frac{p-1}{2}-k'+1)}}{\mk q_K}
\Big)_{\!\! K}\mbox {\ for\ }1\leq k'\leq \frac{p-3}{2},$$
and finally 
$$\Big(\frac{\epsilon_{p-2k'}}{\mk q_K}\Big)_{\!\! K}
=\Big(\frac{\zeta^{\frac{1}{4}-k'^2}}{\mk q_K}
\Big)_{\!\! K}\mbox {\ for\ }1\leq k'\leq \frac{p-3}{2}.$$

(ii) Let us suppose that $q\equiv 1\mod p^2$: then $\Big(\frac{\zeta}{\mk
q_K}\Big)_{\!\! K}=1$ 
and from relation (\ref{e1108211}) we get
$$\Big(\frac{\epsilon_{p-k-1}}{\mk q_K}\Big)_{\!\! K}
=\Big(\frac{\epsilon_{p-k+1}}{\mk q_K}\Big)_{\!\! K}\mbox{\ for\ }
k=2,\dots, p-2.$$
In the other hand we have  $\Big(\frac{\epsilon_{p-1}}{\mk q_K}\Big)_{\!\!
K}
=\Big(\frac{\epsilon_{1}}{\mk q_K}\Big)_{\!\! K}=1$ and so 
$$\Big(\frac{\epsilon_{j}}{\mk q_K}\Big)_{\!\! K}=1\mbox{\  for\
}j=1,\dots,p-1.$$
A straightforward computation shows  that 
$\Big(\frac{\epsilon_1\dots\epsilon_{p-1}}{\mk q_K}\Big)_{\!\! K}
=\Big(\frac{1+\zeta}{\mk q_K}\Big)_{\!\! K}$ and we  
 derive that $$\Big(\frac{1+\zeta}{\mk q_K}\Big)_{\!\! K}=1,$$
and finally that 
$$\Big(\frac{1+\zeta^j}{\mk q_K}\Big)_{\!\! K}=1\mbox{\ for\
}j=1,\dots,p-1.$$
which achieves the proof for $p^2|q-1$.
\end{proof}
\end{thm}

\subsection{Some properties of  the  primes $q$ dividing 
$\frac{(x^p+y^p)(y^p+z^p)(z^p+x^p)}{(x+y)(y+z)(z+x)}$} $ $
We assume that $FLT2$ fails for $(p,x,y,z)$.
In this section, we give,  for a possible future use, 
some general  strong properties  of the primes $q$ dividing  
$\frac{(x^p+y^p)(y^p+z^p)(z^p+x^p)}{(x+y)(y+z)(z+x)}$   in the  second
case of FLT.
Here, we don't assume  that $q$ is $p$-principal or not, 
thus this subsection brings complementary informations to corollary
\ref{c5}. 

Let us define the  totally real cyclotomic units   
$$\varpi_{a}=:\zeta^{(1-a)/2}\cdot\frac{1-\zeta^{a}}{1-\zeta},\ 1 \leq
a\leq p-1,$$
where we observe that $\varpi_1=1$ with this definition.
Recall that the cyclotomic units of $K$ are generated by the $\varpi_a$
for $1<a<\frac{p}{2}$.
With this definition we have $\varpi_a=-\varpi_{p-a}$: indeed we have
$\varpi_a=\zeta^{(1-a)/2}\cdot\frac{1-\zeta^a}{1-\zeta}$ 
and $\varpi_{p-a}=
\zeta^{(1-(p-a))/2}\cdot\frac{1-\zeta^{p-a}}{1-\zeta}=\zeta^{(1+a)/2}\cdot\frac{1-\zeta^{-a}}{1-\zeta}
=\zeta^{1-a)/2}\cdot\frac{\zeta^a-1}{1-\zeta}=-\varpi_a$.

\begin{lem}\label{l1108221}
Assume that FLT2 fails for $(p,x,y,z)$ with $p|y$ . 
Let $\mk q_K$ be a prime  ideal of $\Z_K$ such that $x\zeta+y\equiv 0 \mod
\mk q_K$ (or $z\zeta+  y\equiv 0\mod \mk q_K$).
Then
$$q\equiv 1\mod p^2 \mbox{\ and\ } \Big(\frac{\zeta}{\mk q_K}\Big)_{\!\!
K}=\Big(\frac{p}{\mk q_K}\Big)_{\!\! K}
=\Big(\frac{1-\zeta}{\mk q_K}\Big)_{\!\! K}=1.$$

\begin{proof}$ $

- Suppose that $x\zeta +y\equiv 0\mod \mk q_K$.
We have $q|z$, so $q\equiv 1\mod p^2$ from First Furtwangler's theorem, so
$\Big(\frac{\zeta}{\mk q_K}\Big)_{\!\! K}=1$ and  
$\Big(\frac{x}{\mk q_K}\Big)_{\!\! K}=\Big(\frac{y}{\mk q_K}\Big)_{\!\!
K},$ so
$\Big(\frac{x+z}{\mk q_K}\Big)_{\!\! K}=\Big(\frac{y+z}{\mk
q_K}\Big)_{\!\! K},$   
so 
$\Big(\frac{p^{\nu p-1} y_0^p}{\mk q_K}\Big)_{\!\!
K}=\Big(\frac{x_0^p}{\mk q_K}\Big)_{\!\! K},$  
and finally 
$\Big(\frac{p}{\mk q_K}\Big)_{\!\! K}=1.$
In the other hand, we have 
$$x+y=z_0^p\equiv x(1-\zeta)\equiv (x+z)(1-\zeta)\equiv p^{\nu
p-1}y_0^p(1-\zeta)\mod \mk q_K,$$
so
$$\Big(\frac{1-\zeta}{\mk q_K}\Big)_{\!\! K}=1.$$

- Suppose  that $z\zeta+y\equiv 0\mod \mk q_K$.  The proof  is similar
with $z$ in place of $x$.

\end{proof}
\end{lem}

\begin{lem}\label{l1107011}
Suppose that $FLT2$ fails for $(p,x,y,z)$ with $p|y$ .
Let $q\not= p$ be a prime and  $\mk q_K$ be a  prime ideal of $\Z_K$ over
$q$. Then we have 
for $k=1,\dots,p-2$: 

(i) If $\mk q_K$ divides $x\zeta+y$  then $\Big(\frac{x+\zeta^k y}{\mk
q_K}\Big)_{\!\! K}
=\Big(\frac{\varpi_{k+1}}{\mk q_K}\Big)_{\!\! K}$. 

(ii) If $ \mk q_K$ divides $z\zeta+y$  then $\Big(\frac{z+\zeta^k y}{\mk
q_K}\Big)_{\!\! K}
=\Big(\frac{\varpi_{k+1}}{\mk q_K}\Big)_{\!\! K}$.

(iii) If $\mk q_K$ divides $x\zeta+z$ and $p\ |\ y$   
then $\Big(\frac{x+\zeta^k z}{\mk q_K}\Big)_{\!\! K}\Big(\frac{p }{\mk
q_K}\Big)_{\!\! K}
=\Big(\frac{\varpi_{k+1}}{\mk q_K}\Big)_{\!\! K}$.
\begin{proof}$ $

(i) From  $x\zeta + y\equiv 0\mod \mk q_K$ we get 
 
$$x+\zeta^k y\equiv x(1-\zeta^{k+1})\mod \mk q_K,\ k=1,\dots,p-2.$$
thus 
$$\frac{x+\zeta^k y}{x+y}\equiv \frac{1-\zeta^{k+1}}{1-\zeta}\mod \mk q_K,
\mbox{\ for\ } k=1,\dots,p-2.$$
In the other hand
$\varpi_{k+1}=\zeta^{(1-(k+1))/2}\cdot\frac{1-\zeta^{k+1}}{1-\zeta}$ 
is a totally real cyclotomic unit, 
so  
$$\frac{x+\zeta^k y}{x+y}\equiv \varpi_{k+1}\zeta^{k/2}\mod\mk q_K,\mbox{\
for\ } k=1,\dots p-2,$$ 
so 
$$\Big (\frac{x+\zeta^k y}{\mk q_K}\Big)_{\!\! K}
=\Big(\frac{\varpi_{k+1}}{\mk q_K}\Big)_{\!\!
K}\Big(\frac{\zeta^{k/2}}{\mk q_K}\Big)_{\!\! K} 
\mbox{\ for\ } k=1,\dots,p-2,$$
because $x+y\in K^{\times p}$ and finally
$$\Big (\frac{x+\zeta^k y}{\mk q_K}\Big)_{\!\! K}
=\Big(\frac{\varpi_{k+1}}{\mk q_K}\Big)_{\!\! K}\mbox{\ for\ }
k=1,\dots,p-2,$$
because $q\equiv 1\mod p^2$ obtained by the first Theorem of Furtw\"angler.

(ii) The proof is similar to (i) with $z$ in place of $x$. 

(iii) In that case we have $x+z=p^{\nu p-1}y_0^p$ with $\nu>0$ 
and so $x+z\in p^{-1}K^{\times p}$ and  $p^2|q-1$  as shown in third
paragraph of proof of  lemma \ref{l3}.
\end{proof}
\end{lem}
\begin{rema}
{\rm This  property of the primes $q$ dividing $\frac{x^p+y^p}{x+y}$  
(or  $\frac{z^p+y^p}{z+y}$, or  $\frac{x^p+z^p}{x+z})$   is strong
because $x\zeta+ y$ (or $z\zeta+ y$,
 or $\frac{x+\zeta z}{1-\zeta})$),  
and $\varpi_{k+1}$ of $\Z_K$  
are  pseudo-units not linearly connected in the action of $\F_p[g]$.
}
\end{rema}

\begin{thm}\label{t1107231}
Suppose  that   the second case of FLT fails for $(p,x,y,z)$ with $p|y$.
Let $q$ be a prime dividing $\frac{x^p+y^p}{x+y}$ (or
$\frac{z^p+y^p}{z+y}$ 
or $\frac{x^p+z^p}{x+z}$).
Let $\mk q_K$ be \underline{the} prime ideal of $\Z_K$ over $q$
dividing $x\zeta+ y$ (or $z\zeta+y$ or $x\zeta+z$). 
Assume that  the $p$-class  $c\ell(\mk q_K)\in C\ell^-$.\footnote{See  the paragraph of
notations  \ref{n1} for the meaning of the $p$-class $c\ell$ and the
$p$-class group $C\ell^-$.} 
\footnote{Note that the assumption   $c\ell(\mk q_K)\in C\ell^-$ is automatically 
verified as soon as Vandiver's conjecture holds for $p$.}
Then we have 
\bn
\item 
$q\equiv 1\mod p^2$   and    $q$  is $p$-principal.
\item
$\mk q_K$ verifies the power symbols following values:
\bn
\item 
If $\mk q_K|x\zeta+y$ (or $z\zeta+y$) then $$\Big(\frac{p}{\mk q_K}\Big)_{\!\! K}=\Big(\frac{1-\zeta^j}{\mk q_K}\Big)_{\!\! K}=1\mbox{\ for\ } j=1,\dots,p-1.$$
\item
If $\mk q_K|x\zeta+z$  then $$\Big(\frac{p}{\mk q_K}\Big)_{\!\! K}=\Big(\frac{1-\zeta^j}{\mk q_K}\Big)_{\!\! K} \mbox{\ for\ } j=1,\dots,p-1.$$
\en
\en
\begin{proof}$ $

(a) Suppose at first that $q|\frac{x^p+y^p}{x+y}$: From Furtwangler's
First theorem,  if a prime  $q|\frac{x^p+y^p}{x+y}\frac{y^p+z^p}{y+z}$
then $q\equiv 1\mod p^2$.
We derive that $\Big(\frac{\zeta}{\mk q_K}\Big)_{\!\! K}=1$ and from lemma
\ref{l1108221}  that  $\Big(\frac{p}{\mk q_K}\Big)_{\!\! K}=1$. 

- From   previous lemma \ref{l1107011},  we have 
$$\Big (\frac{x+\zeta^k y}{\mk q_K}\Big)_{\!\! K}
=\Big(\frac{\varpi_{k+1}}{\mk q_K}\Big)_{\!\! K} 
\mbox{\ for\ } k=1,\dots,p-2,$$
and also, with $p-k$ in place of $k$, 
$$\Big (\frac{x+\zeta^{p-k} y}{\mk q_K}\Big)_{\!\! K}
=\Big(\frac{\varpi_{p-k+1}}{\mk q_K}\Big)_{\!\! K} 
\mbox{\ for\ } p-k=1,\dots,p-2.$$

-   The relation (\ref{e1111105}) of theorem \ref{t1107232} can be obtained with exactly the same proof
\be\label{e1111107} 
\Big(\frac{x+\zeta^k y}{\mk q_K}\Big)_{\!\!}
=\Big(\frac{x+\zeta^{p-k}y}{\mk q_K }\Big)_{\!\! K}\mbox{\ for \ }k=2,\dots,p-2,
\ee
which leads to
$$\Big(\frac{\varpi_{k+1}}{\mk q_K}\Big)_{\!\! K}
=\Big(\frac{\varpi_{p-k+1}}{\mk q_K}\Big)_{\!\! K}  \mbox{\ for\ }
k=2,\dots,p-2.$$

- We have  seen above that  $\varpi_{k+1}= -\varpi_{p-k-1}$
so 
$$\Big(\frac{\varpi_{k+1}}{\mk q_K}\Big)_{\!\! K}=
\Big(\frac{\varpi_{p-(k+1)}}{\mk q_K}
\Big)_{\!\! K}\mbox{\ for\ }k=2,\dots,p-2.$$
Then, gathering these two relations,  we get 
$$\Big(\frac{\varpi_{p-k+1}}{\mk q_K}\Big)_{\!\! K}
=\Big(\frac{\varpi_{p-k-1}}{\mk q_K}\Big)_{\!\! K}  \mbox{\ for\ }
k=2,\dots,p-2.$$

-  Starting from $k=2$ we get  for $k=2,4,\dots,p-3,$
$$\Big(\frac{\varpi_{p-1}}{\mk q_K}\Big)_{\!\! K}
=\Big(\frac{\varpi_{p-3}}{\mk q_K}\Big)_{\!\! K} =\dots =
\Big(\frac{\varpi_{2}}{\mk q_K}\Big)_{\!\! K} =1,$$
because we get directly $\Big(\frac{\varpi_{p-1}}{\mk q_K}\Big)_{\!\!
K}=1$ from its definition.
Starting from $k=3$ we get for $k=3,5,\dots, p-2,$
$$\Big(\frac{\varpi_{p-2}}{\mk q_K}\Big)_{\!\! K}
=\Big(\frac{\varpi_{p-4}}{\mk q_K}\Big)_{\!\! K} =\dots =
\Big(\frac{\varpi_{1}}{\mk q_K}\Big)_{\!\! K} =1,$$
because we get directly $\Big(\frac{\varpi_1}{\mk q_K}\Big)_{\!\! K}=1$
from its definition. 
Therefore we get $$\Big(\frac{\varpi_i}{\mk q_K}\Big)_{\!\! K}=1\mbox{\
for\ }i=1,\dots, p-1,$$ and finally
we find again $\Big(\frac{p}{\mk q_K}\Big)_{\!\!
K}=\Big(\frac{1-\zeta}{\mk q_K}\Big)_{\!\! K},$ seen in lemma 
\ref{l1108221}. From this lemma we have also  $\Big(\frac{1-\zeta}{\mk
q_K}\Big)_{\!\! K}=1$ if $\mk q_K|x\zeta+y$ (or $\mk q_K|z\zeta+y$).

- The even  $p$-primary units are all generated by the 
$\varpi_i,\ \ i=1,\dots,\frac{p-1}{2}$.
Therefore,  the result 
$\Big(\frac{\varpi_{i}}{\mk q_K}\Big)_{\!\! K}=1\mbox{\ for\
}i=1,\dots,p-1$ obtained 
and the assumption that $c\ell(\mk q_K)\in C\ell^-$ 
imply  that $\mk q_K$ is $p$-principal  
(application of  the  decomposition  and reflection theorems in the $p$-Hilbert class
field of $K$),
if not it should be possible to find  integers $n_1,\dots,
n_{(p-3)/2}\not\equiv 0\mod p,$
such that the   $p$-primary unit $\varpi=\prod_{i=1}^{(p-3)/2}\varpi_i^{n_i}$ verifies 
$\Big(\frac{\varpi}{\mk q_K}\Big)_{\!\! K}\not=1$, contradiction.

(b)  The proof is similar  if $q|\frac{z^p+y^p}{z+y}$.

(c) Suppose at last  that $q|\frac{x^p+z^p}{x+z}$:
 If $\mk q_K|x\zeta+z$ and $p\ |\ y$   then 
$$\Big(\frac{x+\zeta^k z}{\mk q_K}\Big)_{\!\! K}\Big(\frac{p }{\mk
q_K}\Big)_{\!\! K}
=\Big(\frac{\varpi_{k+1}}{\mk q_K}\Big)_{\!\! K},$$ 
(seen in lemma \ref{l1107011} (iii)) and similarly 
$$\Big(\frac{x+\zeta^{p-k} z}{\mk q_K}\Big)_{\!\! K}\Big(\frac{p }{\mk
q_K}\Big)_{\!\! K}
=\Big(\frac{\varpi_{p-k+1}}{\mk q_K}\Big)_{\!\! K},$$ 
so  we get again
$$\Big(\frac{x+\zeta^k z}{x+\zeta^{p-k} z}\Big)_{\!\! K}
=\Big(\frac{\varpi_{k+1}\varpi_{p-k+1}^{-1}}{\mk q_K}\Big)_{\!\! K}.$$
In the other hand $\frac{x+\zeta^k z}{x+\zeta^{p-k} z}= \zeta^k A$  
where $A$ is a $p$-primary pseudo unit with $A^{s_{-1}+1}\in K^{\times p}$.
Then the end of the proof is similar to the previous cases taking into
account that we know that $p^2|q-1$ from an argument in proof of lemma \ref{l3}  as soon  $q\not=p$ dividing
$\frac{x^p+z^p}{x+z}$, so $\Big(\frac{\zeta^k}{\mk q_K}\Big)_{\!\! K}=1$.
\end{proof}
\end{thm}

\begin{rema}
{\rm 
In the case of an hypothetic solution $(x,y,z),\ p|y$ of the FLT2
equation, 
for the primes $q$ with  $c\ell(\mk q_K)\in C\ell^-$ and $\mk
q_K|x\zeta+y$ (or $z\zeta+y)$,  the theorem \ref{t1107231} can be
considered as  a   reciprocal  statement to corollary \ref{c5}
in which $(u,v)=(x,y)$ or $(z,y)$ for $x,y,z, \ p|y$  hypothetic  solution
of the Fermat's equation.
In particular,  we have proved:

{\it If Vandiver's conjecture holds for $p$,
and if the second case of FLT fails for $p$
then all the primes $q\ |\ \frac{(x^p+y^p)(y^p+z^p)(z^p+x^p)}{(x+y)(y+z)(z+x)}$
are $p$-principal.}}
\end{rema}

\paragraph{ Acknowledgments:}
I would like to thank Georges   Gras  for pointing out many errors   
in the preliminary versions and for  suggesting  many  improvements to me
for  the content and form of the article.

Roland Qu\^eme

13 avenue du ch\^ateau d'eau

31490 Brax

France

mailto: roland.queme@wanadoo.fr
\end{document}